\documentclass[11pt,a4paper]{amsart}
\usepackage{a4wide}

\usepackage[utf8]{inputenc}
\usepackage[english]{babel}

\usepackage{amsfonts,amssymb,amsmath}

\usepackage{xcolor}
\usepackage[colorlinks,
    linkcolor={red!50!black},
    citecolor={blue!50!black},
    urlcolor={blue!80!black}]{hyperref}

\usepackage{tikz}
\usepackage[all]{xy}
\usetikzlibrary{calc, matrix, arrows, cd}

\newtheorem{theorem}{Theorem}[section]

\newtheorem{corollary}[theorem]{Corollary}
\newtheorem{lemma}[theorem]{Lemma}
\newtheorem{proposition}[theorem]{Proposition}

\theoremstyle{definition}
\newtheorem{definition}[theorem]{Definition}

\newtheorem{remark}[theorem]{Remark}

\newtheorem*{claim}{Claim}
\usepackage{enumitem}

\newcommand{\Z}{\mathbb{Z}}

\newcommand{\id}{\operatorname{id}}

\newcommand{\eps}{\varepsilon}
\newcommand{\im}{\operatorname{im}}

\newcommand{\sm}{\setminus}

\DeclareMathOperator{\lk}{lk}

\DeclareMathOperator{\coker}{coker}

\DeclareMathOperator{\Hom}{Hom}
\DeclareMathOperator{\ev}{ev}

\def\op{\operatorname}

\begin{document}

\title{An explicit computation of the Blanchfield pairing for arbitrary links}
\author{Anthony Conway}
\address{Universit\'e de Gen\`eve, Section de math\'ematiques, 2-4 rue du Li\`evre, 1211 Gen\`eve 4, Switzerland}
\email{anthony.conway@unige.ch}
\maketitle
\begin{abstract}
Given a link $L$, the Blanchfield pairing $\op{Bl}(L)$ is a pairing which is defined on the torsion submodule of the Alexander module of $L$. In some particular cases, namely if $L$ is a boundary link or if the Alexander module of $L$ is torsion, $\op{Bl}(L)$ can be computed explicitly; however no formula is known in general. In this article, we compute the Blanchfield pairing of any link, generalizing the aforementioned results. As a corollary, we obtain a new proof that the Blanchfield pairing is hermitian. Finally, we also obtain short proofs of several properties of $\op{Bl}(L)$.
\end{abstract}

\section{Introduction}
The Blanchfield pairing of a knot $K$ is a nonsingular hermitian pairing $\op{Bl}(K)$ on the Alexander module of $K$~\cite{Blanchfield}. Despite early appearances in high dimensional knot theory \cite{KeartonSimple, KeartonSimple2, Ranicki}, the Blanchfield pairing is nowadays mostly used in the classical dimension. For instance, applications of $\op{Bl}(K)$ in knot concordance include a characterization of algebraic sliceness \cite{KeartonCobordism} and a crucial role in the obstruction theory underlying the solvable filtration of \cite{CochranOrrTeichner}, see also \cite{ChaMinimal, CochranHarveyLeidyKnot, FriedlTeichner, Letsche, Litherland}. Furthermore, $\op{Bl}(K)$  has also served to compute unknotting numbers \cite{BorodzikFriedl0, BorodzikFriedl2, BorodzikFriedl} and in the study of finite type invariants \cite{Moussard}. Finally, the Blanchfield pairing can be computed using Seifert matrices \cite{FriedlPowell,KeartonCobordism,Levine}, is known to determine the Levine-Tristram signatures \cite{BorodzikFriedl} and more generally the S-equivalence class of the knot \cite{Trotter}.

In the case of links, the Blanchfield pairing generalizes to a hermitian pairing $\op{Bl}(L)$ on the torsion submodule of the Alexander module of $L$. Although $\op{Bl}(L)$ is still used to investigate concordance \cite{ChaSymmetric,  CochranHarveyLeidy, CochranOrrNotAll, Duval, KimWhitney, Sheiham}, unlinking numbers and splitting numbers \cite{BorodzikFriedlPowell}, several questions remain: is there a natural definition of algebraic concordance for links and can it be expressed in terms of the Blanchfield pairing? Can one compute unlinking numbers and splitting numbers by generalizing the methods of \cite{BorodzikFriedl0, BorodzikFriedl2, BorodzikFriedl} to links? Does the Blanchfield pairing determine the multivariable signature of \cite{CimasoniFlorens}?

A common issue seems to lie at the root of these unanswered questions: there is no general formula to compute the Blanchfield pairing of a link. More precisely, $\op{Bl}(L)$ can currently only be computed if $L$ is a boundary link \cite{HillmanAlexanderIdeals, CochranOrr} or if the Alexander module of $L$ is torsion \cite{ConwayFriedlToffoli}. Note that these formulas generalize the one component case in orthogonal directions: if $L$ is a boundary link whose Alexander module is torsion, then $L$ must be a knot. The aim of this paper is to provide a general formula for the Blanchfield pairing of any colored link, while generalizing the previous formulas.

\medskip
By a~$\mu$-{\em colored link\/}, we mean an oriented link~$L$ in~$S^3$ whose components are partitioned into~$\mu$ sublinks~$L_1\cup\dots\cup L_\mu$. The exterior~$S^3\sm \nu L$ of $L$ will always be denoted by $X_L$. Moreover, we write~$\Lambda_S:=\mathbb{Z}[t_1^{\pm1},\dots,t_\mu^{\pm 1},(1-t_1)^{-1},\dots,(1-t_\mu)^{-1}]$ for the localization of the ring of Laurent polynomials, and we use~$Q=\mathbb{Q}(t_1,\dots,t_\mu)$ to denote the quotient field of~$\Lambda_S$. Using these notations, the Blanchfield pairing of the colored link $L$ is a hermitian pairing
$$ \op{Bl}(L)\colon TH_1(X_L;\Lambda_S) \times TH_1(X_L;\Lambda_S) \,\,\rightarrow\,\, Q/\Lambda_S,$$
where~$TH_1(X_L;\Lambda_S)$ denotes the torsion submodule of the Alexander module $H_1(X_L;\Lambda_S)$ of~$L$, see Section~\ref{sub:Blanchfield} for details. There are two main reasons for which we use $\Lambda_S$ coefficients instead of the more conventional $\Lambda:=\Z[t_1^{\pm1},\dots,t_\mu^{\pm 1}]$ coefficients. The first is to ensure that the Alexander module $H_1(X_L;\Lambda_S)$ admits a square presentation matrix: the corresponding statement is false over $\Lambda$ \cite{CrowellStrauss, Hillman}. The second is to guarantee that  the Blanchfield pairing is non-degenerate after quotienting $TH_1(X_L;\Lambda_S)$ by the so-called \emph{maximal pseudonull submodule}, see~\cite{Hillman}. Note that for knots, the Alexander module over $\Lambda_S$ is the same as the Alexander module over~$\Lambda$ \cite[Proposition 1.2]{Levine}.

As we mentioned above, $\op{Bl}(L)$ can currently only be computed if $L$ is a boundary link, using boundary Seifert surfaces, or if the Alexander module of $L$ is torsion, using some generalized Seifert surfaces known as C-complexes. Let us briefly recall this latter result. A \emph{C-complex} for a $\mu$-colored link $L$ consists in a collection of Seifert surfaces $F_1, \dots , F_\mu$ for the sublinks $L_1, \dots , L_\mu$ that intersect only pairwise along clasps. Given such a C-complex and a sequence~$\varepsilon=(\varepsilon_1,\varepsilon_2,\dots, \varepsilon_\mu)$ of $\pm 1$'s, there are $2^\mu$ generalized Seifert matrices~$A^\varepsilon$ which extend the usual Seifert matrix \cite{CimasoniPotential, CimasoniFlorens, Cooper}. The associated \emph{C-complex matrix} is the $\Lambda$-valued square matrix
$$H\,\,:=\,\,\sum_\varepsilon \,\prod_{i=1}^\mu (1-t_i^{\varepsilon_i}) \, A^\varepsilon,$$
where the sum is on all sequences~$\varepsilon=(\varepsilon_1,\varepsilon_2,\dots, \varepsilon_\mu)$ of~$\pm 1$'s. In \cite[Theorem 1.1]{ConwayFriedlToffoli}, together with Stefan Friedl and Enrico Toffoli, we showed that if $H_1(X_L;\Lambda_S)$ is~$\Lambda_S$-torsion, then the Blanchfield pairing $\op{Bl}(L)$ is isometric to the pairing
\begin{align}
\label{eq:CFT}
\Lambda_S^n / H^T \Lambda_S^n \times \Lambda_S ^n / H^T \Lambda_S^n &\,\,\rightarrow\,\, Q/\Lambda_S \\
(a,b) &\,\,\mapsto\,\, -a^T H^{-1}\overline{b}, \nonumber
\end{align}
where the size $n$ C-complex matrix $H$ for $L$ was required to arise from a \emph{totally connected} C-complex, i.e. a C-complex $F$ in which each $F_i$ is connected and $F_i \cap F_j  \neq \emptyset $ for all $i\neq j$. Note that (\ref{eq:CFT}) also shows that the Alexander module $H_1(X_L;\Lambda_S)$ admits a square presentation matrix. This fact was already known \cite[Corollary 3.6]{CimasoniFlorens}, but as we mentioned above, it is false if we work over $\Lambda$ \cite{CrowellStrauss, Hillman}.

In general, the Blanchfield pairing is defined on the torsion submodule $TH_1(X_L;\Lambda_S)$ of $H_1(X_L;\Lambda_S)$. To the best of our knowledge, this $\Lambda_S$-module has no reason to admit a square presentation matrix, and thus a direct generalization of (\ref{eq:CFT}) seems out of reach. In order to circumvent this issue, we adapt the definition of the pairing described in (\ref{eq:CFT}) as follows.
Note that for any class~$[x]$ in $\text{Tor}_{\Lambda_S}(\Lambda_S^n/H^T\Lambda_S^n)$, there exists an~$x_0$ in $\Lambda_S^n$ such that~$r x=H^Tx_0$ for some $r \in \Lambda_S$. 
As we shall see in Proposition \ref{prop:form}, the assignment~$(v,w)\mapsto \frac{1}{r\overline{s}}v_0^T H\overline{w_0}$ induces a well-defined pairing
$$ \lambda_H\colon \text{Tor}_{\Lambda_S}\big(\Lambda_S^n/H^T\Lambda_S^n\big)\times \text{Tor}_{\Lambda_S}\big(\Lambda_S^n/H^T\Lambda_S^n\big) \to  Q/\Lambda_S,$$
which recovers minus the pairing described in (\ref{eq:CFT}) when~$\text{det}(H) \neq 0$.  Our main theorem reads as follows.

\begin{theorem}
\label{thm:main}
The Blanchfield pairing of a colored link $L$ is isometric to the pairing $-\lambda_H$, where $H$  is any C-complex matrix for $L$.
\end{theorem}

Theorem \ref{thm:main} generalizes \cite[Theorem 1.1]{ConwayFriedlToffoli} to links whose Alexander module $H_1(X_L;\Lambda_S)$ is not torsion and recovers it if $H_1(X_L;\Lambda_S)$ is torsion. Note also that \cite[Theorem 1.1]{ConwayFriedlToffoli} required~$H$ to arise from a \emph{totally connected} C-complex, whereas Theorem \ref{thm:main} removes this extraneous assumption.  As we mentioned above, Theorem \ref{thm:main} also recovers the computation of $\op{Bl}(L)$ when $L$ is a boundary link, see Theorem \ref{thm:Boundary}. Note that to the best of our knowledge, Theorem \ref{thm:main} was not even  known in the case of oriented links (i.e. $\mu=1$) and the result might be of independent interest.

\medskip
While the Blanchfield pairing of a knot is known to be hermitian and nonsingular, the corresponding statements for links require some more care. The hermitian property of $\op{Bl}(L)$ was sorted out by Powell \cite{Powell}, whereas Hillman \cite{Hillman} quotients $TH_1(X_L;\Lambda_S)$ by its \emph{maximal pseudo null submodule} in order to turn $\op{Bl}(L)$ into a non-degenerate pairing, see also \cite[Section 2.5]{BorodzikFriedlPowell}. Even though we avoid discussing the non-degeneracy of the Blanchfield pairing, we observe that Theorem \ref{thm:main} provides a quick proof that $\op{Bl}(L)$ is hermitian. 

\begin{corollary}
\label{cor:NonSingular}
The Blanchfield pairing of a link $L$ is hermitian.
 \end{corollary}

Since the definition of the pairing $\lambda_H$ is quite manageable, we also use Theorem \ref{thm:main} to obtain quick proofs regarding the behavior of $\op{Bl}(L)$ under connected sums, disjoint unions, band claspings, mirror images and orientation reversals, see Proposition \ref{prop:ConnectedSum}, Proposition \ref{prop:BandClaspDisjointUnion} and Proposition \ref{prop:Mirror}. 

\begin{remark}
\label{rem:Scott}
If $T:=TH_1(X_L;\Lambda_S)$ admits a square presentation matrix, then the definition of the pairing $\lambda_H$ can be slightly simplified.
Indeed, in this case the order $\Delta:=\operatorname{Ord}_{\Lambda_S}(T)$ annihilates $T$.
\footnote{Given a $\Lambda_S$-module $M$, for~$k \ge 0$, let~$\Delta^{(k)}(M)$ denote the greatest common divisor of all~$(m-k) \times (m-k)$ minors of an~$m\times n$ presentation matrix of~$M$. 
In other words, $\Delta^{(k)}(M)$ is a generator of the smallest principal ideal containing $E_k(M),$
the $k$-th elementary ideal of $M$.
The order of $M$ is $\operatorname{Ord}_{\Lambda_S}(M)=\Delta^{(0)}(M)$.
If $M$ is torsion and admits a square presentation matrix, then~$\Delta^{(0)}(M)$ annihilates $M$~\cite[Remark 2 on page~31]{HillmanAlexanderIdeals}.
}
Since $T \cong \operatorname{Tors}_{\Lambda_S}(\Lambda_S^n/H^T\Lambda_S^n)$ for any $C$-complex matrix $H$ (as noted in Remark \ref{rem:LambdaSPresentation} below), one therefore has $\lambda_H(v,w)=\frac{1}{\Delta\overline{\Delta}}v_0^TH\overline{w}_0$, for any $v_0,w_0$ with $\Delta v=H^Tv_0$ and $\Delta w=H^Tw_0$.
In other words, one can take $r=s=\Delta$ in the definition of $\lambda_H$.
If, additionally, one chooses a symmetric (up to a sign) representative of~$\Delta$, one can also write~$\Delta^2$ instead of $\Delta\overline{\Delta}$.

Note that $\Delta$ in fact coincides (up to its indeterminacy) with the first non-vanishing Alexander polynomial $\Delta_L^{tor}$ of $L$.
Indeed, setting $r:=\operatorname{rank}_{\Lambda_S}H_1(X_L;\Lambda_S)$, it is known that $\Delta_L^{tor}:=\Delta^{(r)}(H_1(X_L;\Lambda_S))$ equals (up to its indeterminacy) $\Delta:=\operatorname{Ord}_{\Lambda_S}(TH_1(X_L;\Lambda_S))$~\cite[Lemma~4.9]{Turaev}. 
\end{remark}
\color{black}

We conclude this introduction by remarking that Theorem \ref{thm:main} is not a trivial corollary of the work carried out in \cite{ConwayFriedlToffoli}. As we shall see in Section \ref{sec:ProofMain}, removing the torsion assumption on the Alexander module leads to several additional algebraic difficulties. 

\medbreak
This paper is organized as follows. Section \ref{sec:Prelim} briefly reviews twisted homology and the definition of the Blanchfield pairing. Section \ref{sec:ProofMain}, which constitutes the core of this paper, deals with the proof of Theorem \ref{thm:main}. Section \ref{sec:Applications} provides the applications of Theorem \ref{thm:main}.

\subsection*{Notation and conventions.}
We use  $p\mapsto \overline{p}$ to denote the usual involution on $\mathbb{Q}(t_1,\dots,t_\mu)$ induced by $\overline{t_i}=t_i^{-1}$. Furthermore, given a subring $R$ of $\mathbb{Q}(t_1,\dots,t_\mu)$ closed under the involution, and given an $R$-module $M$, we use $\overline{M}$ to denote the $R$-module that has the same underlying additive group as $M$, but for which the action by $R$ on $M$ is precomposed with the involution on $R$. Finally, given any ring $R$,  we think of elements in $R^n$ as column vectors.

\subsection*{Acknowledgments.} 

The author wishes to thank David Cimasoni for his constant flow of helpful advice, Stefan Friedl for suggesting the definition of $\lambda_H$ and Mark Powell for inspiring discussions regarding the definition of the Blanchfield pairing. He also wishes to thank an anonymous referee for several insightful comments, especially regarding the (non-)degeneracy of the Blanchfield pairing. 
The author was supported by the NCCR SwissMap funded by the Swiss FNS.

After this paper was published, Scott Stirling communicated the following mistake to me.
The definition of the pairing $\lambda_H$ originally read as $\lambda_H(v,w)=\frac{1}{\Delta^2}v_0H\overline{w_0}$, with $\Delta$ defined as the order of~$TH_1(X_L;\Lambda_S)$.
Unfortunately $\Delta$ does not necessarily annihilate~$TH_1(X_L;\Lambda_S)$ (see Remark~\ref{rem:Scott}) and so this definition did not make sense in general.
In this corrected version of the published paper, $\lambda_H$ has instead been defined as $\lambda_H(v,w)=\frac{1}{r\overline{s}}v_0H\overline{w_0}$.
The remainder of the paper has been amended accordingly and now differs slightly from the published version.
I am grateful to Scott for pointing out this mistake.

\section{Preliminaries}
\label{sec:Prelim}

This section is organized as follows. Subsection \ref{sub:TwistedHomology} briefly reviews the definition of twisted homology, while Subsection \ref{sub:Blanchfield} gives a definition of the Blanchfield pairing. References include \cite[Section 2]{FriedlLeidyNagelPowell} and \cite[Section 2]{FriedlPowell}.

\subsection{Twisted homology}
\label{sub:TwistedHomology}

Let~$X$ be a CW complex, let~$\varphi\colon \pi_1(X) \rightarrow \mathbb{Z}^\mu$ be an epimorphism, and let~$p\colon \widetilde{X} \rightarrow X$ be the regular cover of~$X$ corresponding to the kernel of~$\varphi$. Given a subspace~$ Y \subset X$, we shall write~$\widetilde{Y}:=p^{-1}(Y)$, and view~$C_*(\widetilde{X},\widetilde{Y})$ as a chain-complex of free left modules over the ring $\Lambda=\mathbb{Z}[t_1^{\pm 1},\dots,t_\mu^{\pm 1}]$. Given a commutative ring $R$ and an $(R,\Lambda)$-bimodule~$M$, consider the chain complexes
\begin{align*}
&C_*(X,Y;M):=M \otimes_{\Lambda}C_*(\widetilde{X},\widetilde{Y}) \\
&C^*(X,Y;M):=\Hom_{\Lambda}\Big(\overline{C_*\big(\widetilde{X},\widetilde{Y}\big)},M\Big)
\end{align*}
of~left $R$-modules and denote the corresponding homology $R$-modules by~$H_*(X,Y;M)$ and~$H^*(X,Y;M)$. Taking $R$ to be $\Lambda_S$ and $M$ to be either $\Lambda_S,Q$ or $Q / \Lambda_S$, we may send a cocycle~$f$ in $\Hom_{\Lambda} \left(  \overline{C_*(\widetilde{X},\widetilde{Y})},M \right)$ to the~$\Lambda_S$-linear map defined by~$\sigma\otimes p \mapsto p \cdot \overline{f(\sigma)}$. This yields a well-defined isomorphism of left $\Lambda_S$-modules  
$$ H^i(X,Y;M) \,\,\rightarrow\,\, H_i\Big(\overline{\Hom_{\Lambda_S}\big(C_*(X,Y;\Lambda_S),M\big)}\Big).$$
We also consider the evaluation homomorphism
$$  H_i\Big(\overline{\Hom_{\Lambda_S}(C_*(X,Y;\Lambda_S),M)}\Big) \,\,\rightarrow\,\, \overline{\Hom_{\Lambda_S}(H_i(C_*(X,Y;\Lambda_S)),M)}.$$
The composition of these two homomorphisms gives rise to the left~$\Lambda_S$-linear map 
$$\text{ev} \colon H^i(X,Y;M) \to \overline{\Hom_{\Lambda_S}(H_i(X,Y;\Lambda_S),M)}.$$
We shall also use repeatedly that the short exact sequence $0 \to \Lambda_S \to Q \to Q/\Lambda_S \to 0$ of coefficients gives rise to the long exact sequence 
\begin{equation}
\label{eq:Bockstein}
 \ldots \to H^k(X,Y;Q) \to H^k(X,Y;Q/\Lambda_S) \to H^{k+1}(X,Y;\Lambda_S) \to H^{k+1}(X,Y;Q) \to \ldots
\end{equation}
in cohomology. The connecting homomorphism $H^k(X,Y;Q/\Lambda_S) \to H^{k+1}(X,Y;\Lambda_S)$ is sometimes referred to as the \emph{Bockstein homomorphism} and will be denoted by $\text{BS}$. Finally, if~$X$ is a compact connected oriented $n$-manifold, there are  Poincar\'e duality isomorphisms $H_i(X,\partial X;M) \cong H^{n-i}(X;M)$ and $H_i(X;M)  \cong H^{n-i}(X,\partial X;M).$

\subsection{The Blanchfield pairing}
\label{sub:Blanchfield}

Let $L=L_1 \cup \dots \cup L_\mu$ be a colored link and denote its exterior by $X_L$. Identifying $\mathbb{Z}^\mu$ with the free abelian group on $t_1,\dots,t_\mu$, the epimorphism~$ \pi_1(X_L) \rightarrow \mathbb{Z}^\mu$ given by $\gamma \mapsto t_1^{\text{lk}(\gamma,L_1)}\cdots t_\mu^{\text{lk}(\gamma,L_\mu)}$ gives rise to the \textit{Alexander module} $H_1(X_L;\Lambda_S)$ of $L$.
Denote by~$\Omega$ the composition
\begin{align*}
TH_1(X_L;\Lambda_S) \,
 & \xrightarrow{(i)} TH_1(X_L,\partial X_L;\Lambda_S)  \\
 & \xrightarrow{(ii)}\, \ker(H^2(X_L;\Lambda_S) \to H^2(X_L;Q))  \\
  & \xrightarrow{(iii)}   \frac{H^1(X_L;Q/\Lambda_S)}{\ker( H^1(X_L;Q/\Lambda_S) \stackrel{\text{BS}}{\to} H^2(X_L;\Lambda_S) )} \\
 & \xrightarrow{(iv)}  \overline{\Hom_{\Lambda_S}(TH_1(X_L;\Lambda_S),Q/\Lambda_S)}
\end{align*}
of the four $\Lambda_S$-homomorphisms defined as follows. The inclusion induced map $H_1(X_L;\Lambda_S) \to H_1(X_L,\partial X_L;\Lambda_S)$ is an isomorphism~\cite[Lemma 2.2]{ConwayFriedlToffoli} and leads to ($i$). Since $H^2(X_L;Q)$ is a $Q$-vector space, torsion elements in $H^2(X_L;\Lambda_S)$ are mapped to zero in $H^2(X_L;Q)$, and thus Poincar\'e duality induces ($ii$). The long exact sequence displayed in (\ref{eq:Bockstein}) implies that the Bockstein homomorphism leads to the homomorphism labeled $(iii)$. Indeed, by exactness $\ker(H^2(X_L;\Lambda_S) \to H^2(X_L;Q))$ is equal to $\im(BS) \cong \frac{H^1(X_L;Q/\Lambda_S)}{\ker(\text{BS})}$. To deal with $(iv)$, we must show that elements of $\ker(BS)$ evaluate to zero on elements of $TH_1(X_L;\Lambda_S)$. Since $\ker(BS)=\im(H^1(X_L;Q) \to H^1(X_L;Q/\Lambda_S))$, elements of $\ker(BS)$ are represented by cocycles which factor through $Q$-valued homomorphisms. Since $Q$ is a field, these latter cocycles vanish on torsion elements, and thus so do the elements of $\ker(BS)$.

\begin{definition}
\label{def:BlanchfieldLinks}
The \textit{Blanchfield pairing} of a colored link~$L$ is the pairing
$$ \operatorname{Bl}(L) \colon  TH_1(X_L;\Lambda_S) \times  TH_1(X_L;\Lambda_S) \rightarrow Q/\Lambda_S ~$$
defined by~$\operatorname{Bl}(L)(a,b)=\Omega(b)(a).$
\end{definition}

It follows from the definitions that the Blanchfield pairing is sesquilinear over~$\Lambda_S$, in the sense that~$\operatorname{Bl}(L)(pa,qb)=p\operatorname{Bl}(L)(a,b)\overline{q}$ for any~$a,b$ in $H_1(X_L;\Lambda_S)$ and any~$p,q$ in $\Lambda_S.$

\section{Proof of Theorem \ref{thm:main}}
\label{sec:ProofMain}

We start by fixing some notation. As we mentioned in the introduction, a C-complex for a $\mu$-colored link $L$ consists in a collection of Seifert surfaces $F_1, \dots , F_\mu$ for the sublinks $L_1, \dots , L_\mu$ that intersect only pairwise along clasps. Pushing a C-complex into the $4$-ball~$D^4$ leads to properly embedded surfaces which only intersect transversally in double points. Let~$W$ be the exterior of such a pushed-in C-complex in~$D^4$, i.e. $W$ is the complement in~$D^4$ of a tubular neighborhood of the pushed-in C-complex, see \cite[Section 3]{ConwayFriedlToffoli} for details. We wish to study the cochain complexes of $\partial W, W$ and $(W,\partial W)$ with coefficients in $\Lambda_S, Q$ and~$Q/\Lambda_S$. These 9 cochain complexes fit in a commutative diagram whose columns and rows are exact. 

\medskip
Keeping this motivating example in mind, we make a short detour which shall only involve homological algebra. More precisely, given a commutative ring $R$, we shall consider the following commutative diagram of cochain complexes of $R$-modules whose columns and rows are assumed to be exact:
\begin{equation}
\label{eq:Nine}
\xymatrix@R0.5cm{
  &0 \ar[d]&  0\ar[d]&  0\ar[d]&  \\
0 \ar[r] &A\ar[d] \ar[r]& B \ar[d]_{v_B}\ar[r]^{h_B}& C \ar[r]\ar[d]& 0 \\
0 \ar[r] &D \ar[r]^{h_D}\ar[d]_{v_D}& E \ar[r]\ar[d]_{v_E}& F \ar[r]\ar[d]& 0 \\
0 \ar[r] &H \ar[r]^{h_H}\ar[d]& J \ar[r]^{h_J}\ar[d]& K \ar[r]\ar[d]& 0  \\
  &0 &  0&  0.& 
}
\end{equation}
We shall write $H^*(D) \to H^*(J)$ for the homomorphism induced by any composition of the cochain maps from $D$ to $J$. Also, $H^*(J) \to H^{*+1}(C)$ will denote the composition of the connecting map from $H^*(J)$ to $H^{*+1}(B)$ with the homomorphism induced by the cochain map from $B$ to $C$. Alternatively, the latter map can also be described as the composition of the homomorphism induced by the cochain map from $J$ to $K$ with the connecting homomorphism $\delta_K^v \colon H^*(K) \to H^{*+1}(C)$. Furthermore, $\delta_K^h$ will denote the connecting homomorphism from~$H^*(K)$ to $H^{*+1}(H)$. Note that these connecting maps are of degree $+1$ since we are working with \emph{co}chain complexes. Finally, we shall use the same notation for cochain maps as for the homomorphisms they induce on cohomology.

We now argue that there is a well-defined homomorphism from $v_D \ker( H^*(D)  \to H^*(J))$ to $H^{*-1}(K)/\ker(\delta_K^h)$, which we shall denote by $(\delta_K^h)^{-1}$. Indeed, if $[x]$ belongs to $\ker( H^*(D)  \to H^*(J))$, the definition of the latter kernel implies that $(h_H \circ v_D)([x])=0$. Using the long exact sequence in cohomology induced by the bottom row of (\ref{eq:Nine}), there is a $[k]$ in $ H^{*-1}(K)$ which satisfies $\delta_K^h([k])=v_D([x])$. Define $(\delta_K^h)^{-1}(v_D([x]))$ as the class of $[k]$ in $H^{*-1}(K)/\ker(\delta_K^h)$. We now check that $(\delta_K^h)^{-1}$ is well-defined. If $[k]$ and $[k']$ are elements of $H^{*-1}(K)$ which satisfy $\delta_K^h([k])=v_D([x])=\delta_K^h([k'])$, then $[k]-[k']$ lies in $\ker(\delta_K^h)$. Consequently, the classes of $[k]$ and $[k']$ agree in the quotient $H^*(B)/\ker(\delta_K^h)$, as desired. 

Similarly, we shall argue that there is a well-defined homomorphism from $h_D \ker(  H^*(D)  \to  H^*(J)) $ to $\frac{ H^*(B)}{\ker(v_B)}$, which we shall denote by $v_B^{-1}$. Indeed, if $[x]$ belongs to $\ker(  H^*(D)  \to  H^*(J))$, the definition of the latter kernel implies that $(v_E \circ h_D)([x])=0$. Using the long exact sequence in cohomology induced by the middle column of (\ref{eq:Nine}), there is a $[b]$ in $ H^*(B)$ which satisfies $v_B([b])=h_D([x])$. Define $v_B^{-1}(h_D([x]))$ as the class of $[b]$ in $\frac{ H^*(B)}{\ker(v_B)}$. We now check that $v_B^{-1}$ is well-defined. If $[b]$ and $[b']$ are elements of $H^{*-1}(B)$ which satisfy $v_B([b])=h_D([x])=v_B([b'])$, then $[b]-[b']$ lies in $\ker(v_B)$. Consequently, the classes of $[b]$ and~$[b']$ agree in the quotient $H^*(B)/\ker(v_B)$, as desired.

Finally, we claim that $\delta_K^v$ induces a well-defined map from $\frac{ H^{*-1}(K)}{\ker(\delta_K^h)} \to \frac{H^*(C)}{\im( H^{*-1}(J) \to  H^*(C))}$. To see this, we must show that if $[k]$ lies in the kernel of $\delta_K^h$, then $\delta_K^v([k])$ belongs to~$\text{im}:=\im( H^{*-1}(J) \stackrel{h_J}{\to}  H^{*-1}(K) \stackrel{\delta_K^v}{\to}  H^*(C))$. By exactness of the bottom row of (\ref{eq:Nine}), we have $\ker(\delta_K^h)=\im(h_J)$. Consequently $[k]$ lies in $\im(h_J)$ and thus $\delta_K^v([k])$ belongs to $\text{im}$, proving the claim.

We delay the proof of the following lemma to the appendix. Note that the statement of this lemma was inspired by \cite[Lemma 4.4]{BargeLannesLatourVogel}.

\begin{lemma}
\label{lem:cube}
Given nine cochain complexes as in (\ref{eq:Nine}), the diagram below anticommutes:
\begin{equation}
\label{eq:9Lemma}
 \xymatrix@R0.5cm{
\ker(  H^*(D)  \to  H^*(J)) \ar[r]^{  v_D}\ar[d]^{h_D} & v_D \ker(  H^*(D)  \to  H^*(J)) \ar[d]^{(\delta_K^h)^{-1}}\\
h_D \ker(  H^*(D)  \to  H^*(J))\ar[d]^{v_B^{-1}} & \frac{H^{*-1}(K)}{\ker(\delta_K^h)}\ar[d]^{\delta_K^v} \\
\frac{ H^*(B)}{\ker(v_B)} \ar[r]^{h_B}& \frac{ H^*(C)}{\im(H^{*-1}(J) \to  H^*(C))}.
}
\end{equation}
\end{lemma}

This concludes our algebraic detour and we now return to topological matters, namely to the nine cochain complexes which arose when we considered the exterior $W$ of a pushed-in C-complex in the $4$-ball.

\medskip
Use $i^W_{\Lambda_S , Q}$ to denote the homomorphism from $H^2(W;\Lambda_S)$ to $H^2(W;Q)$ induced by the inclusion of $\Lambda_S$ into $Q$. We also use $i_{\Lambda_S}^{W,\partial W}$ to denote the homomorphism from $H^2(W;\Lambda_S)$ to $H^2(\partial W;\Lambda_S)$. More generally, we will often implicitly follow this notational scheme: for instance $i_{Q/\Lambda_S}^{(W,\partial W),W}$ will denote the map from $H^2(W,\partial W;Q/ \Lambda_S )$ to $H^2(W;Q/ \Lambda_S)$.

Since $\text{BS}$ plays the role of the boundary map $\delta_K^h$ in our algebraic detour, there is a well-defined map $\text{BS}^{-1}$ from $i_{\Lambda_S}^{W,\partial W}\ker(H^2(W;\Lambda_S) \to H^2(\partial W;Q))$ to $\frac{H^1(\partial W;Q/ \Lambda_S)}{\ker(H^1(\partial W;Q/ \Lambda_S) \stackrel{\text{BS}}{\to} H^2(\partial W;\Lambda_S)}$. Similarly, translating the role of $v_B$ into this setting, there is a well-defined map $(i^{(W,\partial W),W}_Q)^{-1}$ from $i^W_{\Lambda_S,Q} \ker(H^2(W;\Lambda_S) \to H^2(\partial W;Q))$ to $\frac{H^2(W,\partial W;Q)}{\ker(H^2(W,\partial W;Q) \to H^2(W;Q))}$. Furthermore, we shall denote by $\delta_{Q/\Lambda_S}$ the boundary map which arises in the long sequence of the pair $(W,\partial W)$ with $Q/\Lambda_S$ coefficients. 

Applying Lemma \ref{lem:cube} to the cochain complexes of $\partial W$, of $W$ and of $(W,\partial W)$ with coefficients in $\Lambda_S, Q$ and $Q / \Lambda_S$ immediately yields the following lemma.

\begin{lemma}
\label{lem:BigDiagram}
Let $W$ be the exterior of a pushed-in C-complex in $D^4$. 
The following diagram anticommutes:
$$ \xymatrix@R0.6cm@C-1cm{
\ker(H^2(W;\Lambda_S) \to H^2(\partial W;Q)) \ar[rrrrr]^{i_{\Lambda_S}^{W,\partial W}} \ar[d]^{i^W_{\Lambda_S,Q}}
 &&&&& i_{\Lambda_S}^{W,\partial W} \ker(H^2(W;\Lambda_S) \to H^2(\partial W;Q)) \ar[d]^{BS^{-1}} && \\
i^W_{\Lambda_S,Q} \ker(H^2(W;\Lambda_S) \to H^2(\partial W;Q)) \ar[d]^{(i_Q^{(W,\partial W),W})^{-1}}
&&&&& \frac{H^1(\partial W;Q/\Lambda_S)}{\ker(H^1(\partial W;Q/\Lambda_S) \stackrel{BS}{\to} H^2(\partial W;\Lambda_S)) }  \ar[d]^{\delta_{Q/\Lambda_S}} &&   \\ 
\frac{H^2(W,\partial W;Q)}{\ker(H^2(W,\partial W;Q) \to H^2(W;Q))}  \ar[rrrrr]^{i^{(W,\partial W)}_{Q,Q/ \Lambda_S}}
 &&&&& \frac{H^2(W,\partial W;Q/\Lambda_S)}{\im(H^1(\partial W;Q) \to H^2(W,\partial W;Q/\Lambda_S))}.  &&    
}$$
\end{lemma}

Recall from Section \ref{sub:TwistedHomology} that Poincar\'e duality provides isomorphisms from $H_1(\partial W;\Lambda_S)$ to $H^2(\partial W;\Lambda_S)$ and from $H_2(W,\partial W;\Lambda_S)$ to $H^2(W;\Lambda_S)$. Both these maps shall be denoted by~$\text{PD}$. Furthermore, we use $\partial$ to denote the map from $H_2(W,\partial W;\Lambda_S)$ to $H_1(\partial W;\Lambda_S)$ which arises in the long exact sequence of the pair $(W,\partial W)$. We shall abbreviate $TH_1(\partial W;\Lambda_S)$ by~$T$. Finally, we recall that a C-complex $F = F_1 \cup \ldots \cup F_\mu$ is \emph{totally connected} if each $F_i$ is connected and $F_i \cap F_j  \neq \emptyset $ for all $i\neq j$.

\begin{lemma}
\label{lem:PDBienDef}
Let $W$ be the exterior of a pushed-in C-complex in $D^4$.
\begin{enumerate}
\item Poincar\'e duality restricts to a well-defined map $ \partial^{-1}(T) \to  \ker(H^2(W;\Lambda_S) \to H^2(\partial W;Q)).$
\item If the C-complex is totally connected, then Poincar\'e duality restricts to a well-defined map  $ T  \to   i_{\Lambda_S}^{W,\partial W} \ker(H^2(W;\Lambda_S) \to H^2(\partial W;Q)).$ 
\end{enumerate}
\end{lemma}
\begin{proof}
In order to prove both statements, we shall consider the following commutative diagram:
\begin{equation}
\label{eq:DiagPDBiendef}
\xymatrix{
H_2(W,\partial W;\Lambda_S) \ar[r]^{\text{PD}} \ar[d]^{\partial}& H^2(W;\Lambda_S) \ar[r]^{i^W_{\Lambda_S,Q}}\ar[d]^{i_{\Lambda_S}^{W,\partial W}}& H^2(W;Q) \ar[d]^{i_Q^{W,\partial W}} \\
H_1(\partial W;\Lambda_S) \ar[r]^{\text{PD}} & H^2(\partial W;\Lambda_S) \ar[r]^{i^{\partial W}_{\Lambda_S,Q}}& H^2(\partial W;Q). \\
}
\end{equation}
We start with the first assertion. Given $x$ in $\partial^{-1}(T)$, the goal is to show that $PD(x)$ lies in $\ker(H^2(W;\Lambda_S) \to H^2(\partial W;Q))$ or in other words, we wish to show that $(i^{\partial W}_{\Lambda_S,Q} \circ  i_{\Lambda_S}^{W,\partial W} \circ PD)(x)$ vanishes. Since $\partial (x)$ is a torsion element of $H_1(\partial W;\Lambda_S)$, there exists a non-zero $\lambda$ in $\Lambda_S$ for which $\lambda \partial (x)=0$. The commutativity of (\ref{eq:DiagPDBiendef}) now implies that $\lambda (i^{\partial W}_{\Lambda_S,Q} \circ i_{\Lambda_S}^{W,\partial W} \circ PD)(x)=(i^{\partial W}_{\Lambda_S,Q} \circ PD)(\lambda \partial( x))=0$. Since $H^2(W;Q)$ is a vector space and $\lambda$ is non-zero, the first claim is proved.

Next, we deal with the second claim. Given $a$ in $T$, we must find a $d$ in $\ker(H^2(W;\Lambda_S) \to H^2(\partial W;Q))$ such that $i_{\Lambda_S}^{W,\partial W}(d)=PD(a)$. Since we now assume the C-complex to be totally connected, \cite[Corollary 3.2]{ConwayFriedlToffoli} implies that $H_1(W;\Lambda_S)=0$ and thus $\partial$ is surjective. Consequently, there exists an $x$ in $H_2(W,\partial W;\Lambda_S)$ for which $\partial (x)=a$. Since $a$ is torsion, $x$ is actually in $\partial^{-1}(T)$ and so the first claim implies that $PD(x)$ lies in $\ker(H^2(W;\Lambda_S) \to H^2(\partial W;Q))$. Thus we set $d:=PD(x)$ and observe that the commutativity of (\ref{eq:DiagPDBiendef}) implies $PD(a)=PD(\partial(x))=i_{\Lambda_S}^{W,\partial W}(PD(x))=i_{\Lambda_S}^{W,\partial W} (d)$, as desired.
\end{proof}

Next, we deal with the evaluation maps which were described in Section \ref{sub:TwistedHomology}. More precisely we shall consider the map from $H^2(W,\partial W;Q)$ to $ \overline{ \Hom_{\Lambda_S}(H_2(W,\partial W;\Lambda_S),Q) }$ and the map from $H^2(W,\partial W;Q/ \Lambda_S)$ to~$\overline{ \Hom_{\Lambda_S}(H_2(W,\partial W;\Lambda_S),Q/ \Lambda_S) }$.

\begin{lemma}
\label{lem:EvBienDef}
Let $W$ be the exterior of a pushed-in C-complex in $D^4$.
\begin{enumerate}
\item The evaluation map on $H^2(W,\partial W;Q)$ induces a well-defined map 
$$\text{ev} \colon \frac{H^2(W,\partial W;Q)}{\ker(H^2(W,\partial W;Q) \to H^2(W;Q)) } \to  \overline{ \Hom_{\Lambda_S}(\partial^{-1}(T),Q) }.$$
\item The evaluation map on $H^2(W,\partial W;Q/ \Lambda_S)$ induces a well-defined map 
 $$\text{ev} \colon \frac{H^2(W,\partial W;Q/\Lambda_S)}{\im(H^1(\partial W;Q) \to H^2(W,\partial W;Q/\Lambda_S))} \to   \overline{ \Hom_{\Lambda_S}(\partial^{-1}(T),Q/\Lambda_S) }.$$
\end{enumerate}
\end{lemma}
\begin{proof}
From now on, we shall write $\langle \varphi,x \rangle$ instead of $(\text{ev})(\varphi)(x)$. We start by proving the first assertion. First of all, by exactness we have $\ker(H^2(W,\partial W;Q) \to H^2(W;Q))=\im(H^1(\partial W;Q) \stackrel{\delta_Q}{\to} H^2(W,\partial W;Q))$, where $\delta_Q$ denotes the boundary map in the long exact sequence of the pair. Consequently, the goal is to show that for all $\varphi$ in $H^1(\partial W;Q)$ and all $x$ in $\partial^{-1}(T)$, one has $\langle \delta_Q \varphi,x \rangle=0$. Consider the following commutative diagram:
\begin{equation}
\label{eq:EvBienDef}
\xymatrix{
H^1(\partial W;Q) \ar[d]^{\delta_Q} \ar[rrr]^{\text{ev}} &&& \overline{\op{Hom}_{\Lambda_S}(H_1(\partial W;\Lambda_S),Q)} \ar[d]^{\partial^*} \\
H^2(W,\partial W;Q) \ar[rrr]^{\text{ev}}&&& \overline{\op{Hom}_{\Lambda_S}(H_2(W,\partial W;\Lambda_S),Q)}.}
\end{equation}
Since $\partial x$ is torsion, there exists a non-zero $\lambda$ in $\Lambda_S$ for which $\lambda \partial (x)$ vanishes. The diagram in~(\ref{eq:EvBienDef}) now gives $\lambda \langle \delta_Q \varphi,x \rangle=\lambda \langle \varphi,\partial x \rangle =\langle \varphi,\lambda \partial (x) \rangle =0$. Since this equation takes place in the field $Q$ and $\lambda$ is non-zero, we get $\langle  \delta_Q \varphi,x \rangle=0$, as desired.

To prove the second claim, start with $\varphi$ in $H^1(\partial W;Q)$ and $x$ in $\partial^{-1}(T)$. Consider the change of coefficient homomorphism $i_{Q,Q/\Lambda_S}^{\partial W} \colon H^1(\partial W;Q) \to H^1(\partial W;Q/\Lambda_S)$ and the connecting  homomorphism $\delta_{Q/\Lambda_S} \colon H^1(\partial W;Q/ \Lambda_S) \to H^2(W,\partial W;Q/\Lambda_S)$. In order to show that $\langle (\delta_{Q/\Lambda_S} \circ i_{Q,Q/\Lambda_S}^{\partial W})(\varphi), x \rangle=0$, consider the same commutative diagram as displayed in~(\ref{eq:EvBienDef}) but with~$Q/\Lambda_S$ coefficients:
\begin{equation}
\label{eq:EvBienDefQLambda}
\xymatrix{
H^1(\partial W;Q/\Lambda_S) \ar[d]^{\delta_{Q/\Lambda_S}} \ar[rrr]^{\text{ev}} &&& \overline{\op{Hom}_{\Lambda_S}(H_1(\partial W;\Lambda_S),Q/\Lambda_S)} \ar[d]^{\partial^*} \\
H^2(W,\partial W;Q/\Lambda_S) \ar[rrr]^{\text{ev}}&&& \overline{\op{Hom}_{\Lambda_S}(H_2(W,\partial W;\Lambda_S),Q/\Lambda_S)}.}
\end{equation}
 Since $\varphi$ is $Q$-valued and $\partial(x)$ is torsion, the result follows from the commutativity of (\ref{eq:EvBienDefQLambda}). Indeed, $\langle (\delta_{Q/\Lambda_S} \circ i_{Q,Q/\Lambda_S}^{\partial W})(\varphi), x \rangle=\langle (i_{Q,Q/\Lambda_S}^{\partial W})(\varphi), \partial(x) \rangle$ and the latter term vanishes since cocycles which factor through $Q$ vanish on torsion elements.
\end{proof}

Recall that we use $\text{BS}^{-1}$ to denote the map from $i_{\Lambda_S}^{W,\partial W}\ker(H^2(W;\Lambda_S) \to H^2(\partial W;Q))$ to $\frac{H^1(\partial W;Q/ \Lambda_S)}{\ker(H^1(\partial W;Q/ \Lambda_S) \to H^2(\partial W;\Lambda_S)}$ which appeared in Lemma \ref{lem:BigDiagram}. Combining the previous results, we obtain the following lemma.

\begin{lemma}
\label{lem:Combine}
Let $L$ be a colored link and let $W$ be the exterior of a pushed-in totally connected C-complex for $L$. The squares and triangle in the following diagram commute, while the top pentagon anticommutes. Furthermore, the map $\Gamma:=\text{ev} \circ \text{BS}^{-1} \circ \text{PD}$ coincides with the adjoint of the Blanchfield pairing $\op{Bl}(L)$.
\begin{equation}
\label{eq:Combine}
\xymatrix@R0.6cm@C-1cm{
\partial^{-1}(T) \ar[rrrrrrrrrrrrr]^-\partial \ar[dd]_{(i_Q^{(W,\partial W),W})^{-1} \circ i_{\Lambda_S,Q}^W\circ \text{PD}}  &&&&&&&&&&&&& T  \ar[d]^{\text{BS}^{-1} \circ \text{PD}} 
\ar@/^2pc/[ddrr]^{\Gamma} 
&&\\
&&&&&&&&&&&&& \frac{H^1(\partial W;Q/\Lambda_S)}{\ker(BS)} \ar[d]^{\delta_{Q/\Lambda_S}} \ar[rrd]^{\text{ev}}\\
\frac{H^2(W,\partial W;Q)}{\ker(H^2(W,\partial W;Q) \to H^2(W;Q)) }  \ar[rrrrrrrrrrrrr]^{i_{Q,Q/\Lambda_S}^{(W,\partial W)}} \ar[d]_{\text{ev}}  &&&&&&&&&&&&& \frac{H^2(W,\partial W;Q/\Lambda_S)}{\im(H^1(\partial W;Q) \to H^2(W,\partial W;Q/\Lambda_S))} \ar[rrd]^{\text{ev}}  &&   \overline{ \Hom_{\Lambda_S}(T,Q/\Lambda_S) } \ar[d]^{\partial^*} \\
\overline{ \Hom_{\Lambda_S}(\partial^{-1}(T),Q) }\ar[rrrrrrrrrrrrrrr] &&&&&&&&&&&&&&&  \overline{ \Hom_{\Lambda_S}(\partial^{-1}(T),Q/\Lambda_S) }.
}
\end{equation}
\end{lemma}
\begin{proof}
We start by arguing that the maps in (\ref{eq:Combine}) are well-defined. For the upper right evaluation map, this follows from the same argument as the one which was used in Section~\ref{sub:Blanchfield}, just before Definition \ref{def:BlanchfieldLinks}. All the other maps are well-defined thanks to Lemma~\ref{lem:BigDiagram}, Lemma \ref{lem:PDBienDef} and Lemma \ref{lem:EvBienDef}.  The top pentagon anticommutes thanks to Lemma~\ref{lem:PDBienDef} and Lemma~\ref{lem:BigDiagram}. 
The top triangle commutes by definition of $\Gamma$, the bottom square clearly commutes, while the commutativity of the rightmost square follows from~(\ref{eq:EvBienDefQLambda}). To prove the second assertion, we start by noting that \cite[Lemma 5.2]{ConwayFriedlToffoli} implies that the inclusion induced map $H_1(X_L;\Lambda_S) \to H_1(\partial W;\Lambda_S)$ is an isomorphism. Using this fact, we observe that $\Gamma$ is defined exactly as the adjoint $\Omega$ of the Blanchfield pairing was, see Subsection \ref{sub:Blanchfield}.
\end{proof}

Looking at the leftmost column of (\ref{eq:Combine}), we wish to define a pairing on $\partial^{-1}(T)$. To do this, we start by considering the composition
\begin{align*}
\Theta \colon \ \ \ \partial^{-1}(T)
&\stackrel{\text{PD}}{\longrightarrow} \ker(H^2(W;\Lambda_S) \to H^2(\partial W;Q))   \\
&\stackrel{i^W_{\Lambda_S,Q}}{\longrightarrow} i^W_{\Lambda_S,Q} \ker(H^2(W;\Lambda_S) \to H^2(\partial W;Q))\\
&\longrightarrow  \frac{H^2(W,\partial W;Q)}{\ker(H^2(W,\partial W;Q) \to H^2(W;Q)) } \\
&\stackrel{\text{ev}}{\longrightarrow}  \overline{ \Hom_{\Lambda_S}(\partial^{-1}(T),Q) }
\end{align*}
of $\Lambda_S$-linear homomorphisms, where the third arrow denotes the homomorphism $(i_Q^{(W,\partial W),W})^{-1}$ which was described in the discussion leading up to Lemma \ref{lem:BigDiagram}. Note that the first map is well-defined thanks to Lemma \ref{lem:PDBienDef}, the second map is obviously well-defined, the discussion prior to Lemma \ref{lem:BigDiagram} ensures that the third map is well-defined, and the fourth map is well-defined thanks to Lemma \ref{lem:EvBienDef}. We define the desired pairing on $ \partial^{-1}(T)$ by 
$$\theta(x,y):=\Theta(y)(x).$$
Recall from Lemma \ref{lem:Combine} and its proof that the pairing defined by $\Gamma$ on $TH_1(\partial W;\Lambda_S)$ coincides with the Blanchfield pairing on $TH_1(X_L;\Lambda_S)$. Using these identifications, Lemma \ref{lem:Combine} implies the following proposition.

\begin{proposition}
\label{prop:ReduceToTheta}
Let $L$ be a colored link and let $W$ be the exterior of a pushed-in totally connected C-complex for $L$. The following diagram commutes:
\begin{equation}
\label{eq:DiagramThreeLines}
\xymatrix@C1.4cm@R0.5cm{ 
\partial^{-1}(TH_1(\partial W;\Lambda_S)) \times \partial^{-1}(TH_1(\partial W;\Lambda_S)) \ar[r]^{\ \ \ \ \ \ \ \ \ \ \ \ \ \ \ \ \ \ \ \  -\theta}\ar[d]^{\partial \times \partial } & Q \ar[d]   \\
TH_1(\partial W;\Lambda_S) \times TH_1(\partial W;\Lambda_S) \ar[r]^{\ \ \ \ \ \ \ \ \ \  \ \ \operatorname{Bl}(L)}& Q/\Lambda_S.
}
\end{equation} 
\end{proposition}

As (\ref{eq:DiagramThreeLines}) suggests, the computation of the Blanchfield pairing now boils down to the computation of $\theta$. The remainder of the proof is devoted to this task.

\medskip
From now on, we shall assume that $W$ is the exterior of a pushed-in \emph{totally connected} C-complex. The intersection form $\lambda$ on $W$ is defined as the adjoint of the composition
\begin{equation}
\label{eq:DefIntersection}
\Phi \colon H_2(W;\Lambda_S) \stackrel{i}{\to} H_2(W,\partial W;\Lambda_S) \stackrel{\text{PD}}{\rightarrow} H^2(W;\Lambda_S) \stackrel{\text{ev}}{\rightarrow} \overline{\op{Hom}_{\Lambda_S}(H_2(W;\Lambda_S),\Lambda_S)}.
\end{equation}
In other words, $\lambda(x,y):=\Phi(y)(x)$, see for instance~\cite[Section 2.3]{ConwayFriedlToffoli} for details. In particular, we notice that $\Phi$ vanishes on $\ker(i)$ and descends to a map on $H_2(W;\Lambda_S)/\ker(i)$ which we also denote by $\Phi$.

Since we assumed that $W$ is the exterior of a pushed-in totally connected C-complex, \cite[Corollary 3.2]{ConwayFriedlToffoli} implies that $H_1(W;\Lambda_S)=0$. Thus, there is an exact sequence 
$$H_2(W;\Lambda_S) \stackrel{i}{\to} H_2(W,\partial W;\Lambda_S) \stackrel{\partial}{\to} H_1(\partial W;\Lambda_S) \to 0.$$
Consequently, from now on, we shall identify $H_1(\partial W;\Lambda_S)$ with the cokernel of the map $i$. In particular, elements of $H_1(\partial W;\Lambda_S)$ will be denoted by $[x]$, where $x$ lies in $H_2(W,\partial W;\Lambda_S)$. Furthermore, we shall identify the boundary map $\partial$ with the quotient map of $H_2(W,\partial W ; \Lambda_S)$ onto $\coker(i)$. In other words, we allow ourselves to interchangeably write $\partial (x)$ and $[x]$. 

Let $x,y$ be in $\partial^{-1}(T)$. Since $[x]$ and $[y]$ are torsion, there exists $x_0$ and $y_0$ in $H_2(W;\Lambda_S)$ and $r,s \in \Lambda_S$ such that $r x=i(x_0)$ and $s y =i(y_0)$. 
Define a $Q$-valued pairing $\psi$ on $\partial^{-1}(T)$ by setting
$$\psi(x,y) := \frac{1}{r\overline{s}}\lambda(x_0,y_0).$$
Observe that $\psi$ is well-defined: if $x_0$ and $x_0'$ satisfy $i(x_0)=r x$ and $i(x_0')=r'x$, then $r'x_0-rx_0'$ lies in $\ker(i)$ and thus 
$\frac{r'}{r'r\overline{s}}\lambda(x_0,y)-\frac{r}{rr'\overline{s}}\lambda(x_0',y)=\frac{1}{r'r\overline{s}}\lambda(r'x_0-rx_0',y)=0,$
as we observed above. 
The same reasoning applies to the second variable. In particular, we could have very well taken $x_0$ and $y_0$ in $H_2(W;\Lambda_S)/\ker(i)$. Summarizing, we have two $Q$-valued pairings defined on $\partial^{-1}(T)$ and we wish to show that they agree:

\begin{proposition}
\label{prop:ComputeTheta}
$\theta$ is equal to $\psi$.
\end{proposition}

Before diving into the proof, we set up some notation. First, define $j \colon \partial^{-1}(T) \to \im(i)\otimes_{\Lambda_S} Q$ as follows. 
Given $x$ in $\partial^{-1}(T)$, we set $j(x):=i(x_0) \otimes \frac{1}{r}$, where $x_0$ is any element of $H_2(W;\Lambda_S)$ which satisfies $i(x_0)=r x$ with $r \in \Lambda_S$. 
The map $j$ is well-defined as $i(x_0) \otimes \frac{1}{r}=rx \otimes \frac{1}{r}=x \otimes 1$.
Next, we set 
$$K:=\ker (H^2(W;\Lambda_S) \stackrel{i^{W,\partial W}_{\Lambda_S}}{\longrightarrow} H^2(\partial W;\Lambda_S) \stackrel{i^{\partial W}_{\Lambda_S,Q}}{\longrightarrow} H^2(\partial W;Q)).$$ 
Note that $K$ already appeared in Lemma \ref{lem:BigDiagram} as well as in the definition of $\theta$. The discussion leading up to Lemma \ref{lem:BigDiagram} also provided a homomorphism $(i_Q^{(W,\partial W),W})^{-1}$ whose domain was $i^W_{\Lambda_S,Q}(K)$. 
For the moment however, we shall rename it as 
$$k^* \colon \ \ i^W_{\Lambda_S,Q}(K) \to \frac{H^2(W,\partial W;Q)}{\ker(H^2(W,\partial W ;Q) \to H^2(W;Q))} $$
and recall its definition. Given $\phi$ in $K$, the definition of $K$ implies that $(i_Q^{\partial W, W} \circ i^W_{\Lambda_S,Q})(\phi)$ vanishes. Using the exactness of the long exact sequence of the pair $(W,\partial W)$ with $Q$ coefficients, it follows that $i_{Q}^{(W,\partial W),W}(\xi)=i^W_{\Lambda_S,Q}(\phi)$ for some $\xi \in H^2(W,\partial W;Q)$. The map $k^*$ is defined by $k^* (i^W_{\Lambda_S,Q} (\phi))=[\xi]$.

\begin{remark}
\label{rem:ExampleOfMapk}
Note that if $\phi= i_{\Lambda_S}^{(W,\partial W),W}( \varphi)$ for some $\varphi$ in $H^2(W,\partial W;\Lambda_S)$, then the description of $k^*$ becomes more concrete. The reason is that we can pick $\xi$ to be $ i_{\Lambda_S,Q}^{(W,\partial W)}(\varphi)$. Indeed, we have 
$$i_{Q}^{(W,\partial W),W}(\xi)
=(i_{Q}^{(W,\partial W)} \circ i_{\Lambda_S,Q}^{(W,\partial W)})(\varphi)
= ( i^W_{\Lambda_S,Q} \circ i_{\Lambda_S}^{(W,\partial W),W})( \varphi) 
=i^W_{\Lambda_S,Q}(\phi),$$
where the second equality follows from the diagram below:
\begin{equation}
\label{eq:ExampleOfMapk}
\xymatrix{
H^2(W,\partial W;\Lambda_S) \ar[r]^{i_{\Lambda_S,Q}^{(W,\partial W)}} \ar[d]^{i_{\Lambda_S}^{(W,\partial W),W}}
& H^2(W,\partial W;Q) \ar[d]^{i_Q^{(W,\partial W),W}}\\
H^2(W;\Lambda_S) \ar[r]^{i_{\Lambda_S,Q}^{W}} 
 & H^2(W;Q).
}
\end{equation}
Summarizing, we have~$(k^* \circ i^W_{\Lambda_S,Q} \circ i_{\Lambda_S}^{(W,\partial W),W})(\varphi)= i_{\Lambda_S,Q}^{(W,\partial W)}(\varphi)$. 
\end{remark}

Let us temporarily write $V$ instead of $H_2(W;\Lambda_S)$. Proposition \ref{prop:ComputeTheta} will follow if we manage to show that all the maps in (\ref{eq:PsiEqualsTheta}) are well-defined and produce a commutative diagram. Indeed, in this diagram, there are several routes which lead from the upper right corner to the lower left corner. Taking the uppermost route produces the pairing $\psi$, while the lowermost route produces $\theta$:

\begin{equation}
\label{eq:PsiEqualsTheta}
\xymatrix@R1cm@C0.45cm{
\frac{V}{\ker(i)} \otimes_{\Lambda_S} Q \ar[dd]^{\widetilde{\Phi}} \ar[r]^{i \otimes \id_Q}_\cong
& \im(i) \otimes_{\Lambda_S} Q \ar[d]^{PD \otimes  \id_Q}
& \partial^{-1}(T) \ar[l]_j \ar[d]^{\text{PD}} \\
&i_{\Lambda_S}^{(W,\partial W),W} \circ \text{PD} (V)  \otimes_{\Lambda_S} Q \ar[d]^{i^W_{\Lambda_S,Q} \otimes \id_Q}
& K \ar[d]^{i^W_{\Lambda_S,Q}} \\
\overline{\op{Hom}_{\Lambda_S}(\frac{V}{\ker(i)} \otimes_{\Lambda_S} Q,Q)}  \ar[d]^{  j^*((i^{-1})^*\otimes Q)  }
&  i^W_{\Lambda_S,Q} \circ i_{\Lambda_S}^{(W,\partial W),W} \circ PD(V)\otimes_{\Lambda_S} Q \ar[l]_-{{\widetilde{ev}}}  \ar[d]^{k^* \otimes \id_Q}  \ar[r]^-{\text{mult}}
&i^W_{\Lambda_S,Q}(K)  \ar[d]^{k^*}\\
\overline{\op{Hom}_{\Lambda_S}(\partial^{-1}(T),Q)} 
&  k^* \circ i^W_{\Lambda_S,Q} \circ i_{\Lambda_S}^{(W,\partial W),W} \circ PD(V) \otimes_{\Lambda_S} Q  \ar[r]^-{\text{mult}}  \ar[l]_-{{\widetilde{ev}}}
& \frac{H^2(W,\partial W;Q)}{\ker(H^2(W,\partial W;Q) \to H^2(W;Q))}. \ar@/^2pc/[ll]^{\text{ev}}
}
\end{equation}

We define the maps $\text{mult},\widetilde{\ev}$ and $\widetilde{\Phi}$.
In both instances, $\text{mult}(\varphi \otimes q)=q \cdot \varphi$.
The bottom left evaluation map is given by~$\widetilde{\ev}(\varphi \otimes q)(x)=\overline{q}\langle \varphi,x\rangle.$
The middle evaluation map is given by~$\widetilde{\ev}(\varphi \otimes b)(x \otimes a)=a\overline{b}\langle \varphi,x\rangle.$
The map labeled $\widetilde{\Phi}$ is defined as~$\widetilde{\Phi}([y_0] \otimes b)([x_0] \otimes a)=a\overline{b} \Phi([y_0])([x_0])$, where~$\Phi$ (the adjoint of the intersection form) was defined in~\eqref{eq:DefIntersection}.
\color{black}

We now argue that all the maps in (\ref{eq:PsiEqualsTheta}) are well-defined. 
We already checked that the rightmost vertical maps are well-defined, see Lemma \ref{lem:BigDiagram} and Lemma \ref{lem:PDBienDef}.
The middle Poincar\'e duality map is well-defined: this follows immediately from the equality $PD \circ i =i_{\Lambda_S}^{(\partial W, W),W} \circ PD.$ 
Next, we deal with the two multiplication maps on the bottom right. 
First observe that $i_{\Lambda_S}^{(\partial W, W),W} \circ PD(V)$ is a subspace of $K$: indeed $K=\ker (H^2(W;\Lambda_S) \stackrel{i^{W,\partial W}_{\Lambda_S}}{\longrightarrow} H^2(\partial W;\Lambda_S) \stackrel{i^{\partial W}_{\Lambda_S,Q}}{\longrightarrow} H^2(\partial W;Q))$, and $i^{W,\partial W}_{\Lambda_S} \circ i_{\Lambda_S}^{(W,\partial W),W}=0$ by exactness. 
Consequently $i^W_{\Lambda_S,Q} \circ  i_{\Lambda_S}^{(W,\partial W),W} \circ PD(V)$ is a subspace of $i^W_{\Lambda_S,Q}(K)$. It then follows that $k^* \circ i^W_{\Lambda_S,Q} \circ i_{\Lambda_S}^{(W,\partial W),W} \circ PD(V)$ is a subspace of the term in the lower right corner. 
Therefore the maps labeled ``mult" are well defined.
It also follows from these observations and Lemma \ref{lem:EvBienDef} that the lower two evaluation maps in (\ref{eq:PsiEqualsTheta}) are well-defined. 
The next lemma will conclude the proof of Proposition \ref{prop:ComputeTheta}.

\begin{lemma}
\label{lem:AllCommute}
All the squares in (\ref{eq:PsiEqualsTheta}) commute.
\end{lemma}
\begin{proof}
The upper left square commutes by definition of $\Phi$, see (\ref{eq:DefIntersection}). 
The bottom right square and the bottom triangle clearly commute. 
Let us now deal with the large rectangle on the upper right. Start with $x$ in $\partial^{-1}(T)$. 
Using the definition of $j$, we have $j(x)=i(x_0) \otimes \frac{1}{r}$, where $r \in \Lambda_S$ and $x_0 \in H_2(W;\Lambda_S)$ satisfy $i(x_0)=r x$. The desired relation now follows readily:
$$\text{mult}(i^W_{\Lambda_S,Q} \circ \text{PD}  \circ j)(x)=\frac{1}{r} (i^W_{\Lambda_S,Q} \circ \text{PD} \circ i)(x_0)=(i^W_{\Lambda_S,Q} \circ PD)(x).$$
Finally, we deal with the lower left square. 
Let $\varphi$ be in $H^2(W,\partial W;\Lambda_S)$, let $q \in Q$ and let $x$ be in $\partial^{-1}(T)$. 
Using once again the definition of $j$, we have $(i^{-1}\circ j)(x)=[x_0]\otimes \frac{1}{r}$ where $r \in \Lambda_S$ and $x_0 \in H_2(W;\Lambda_S)$ satisfy $i(x_0)=r x$. Consequently, writing brackets in place of~$\widetilde{ev}$, we get the relation
\begin{align*}
 \langle  (i^W_{\Lambda_S,Q} \circ  i_{\Lambda_S}^{(W,\partial W),W})(\varphi) \otimes q , (i^{-1} \circ j)(x) \rangle 
&=\overline{q}  \langle  (i^W_{\Lambda_S,Q} \circ  i_{\Lambda_S}^{(W,\partial W),W})(\varphi) , [x_0] \otimes \frac{1}{r} \rangle
=\overline{q} \frac{1}{r} \langle  \varphi , i ([x_0]) \rangle \\
&=\overline{q}\langle \varphi,x \rangle,
\end{align*}
where in the second equality, we simultaneously used that induced maps commute with evaluations and the fact that $i_{\Lambda_S,Q}^W$ changes the coefficients without affecting the expression involved. 
On the other hand, recalling the conclusion of Remark \ref{rem:ExampleOfMapk}, we can compute the other term:
$$ \langle (k^* \circ i^W_{\Lambda_S,Q} \circ  i_{\Lambda_S}^{(W,\partial W),W})(\varphi) \otimes q , x\rangle
=\overline{q}\langle i_{\Lambda_S,Q}^{(W,\partial W)}(\varphi), x \rangle
=\overline{q} \langle  \varphi ,  x \rangle.$$
Combining these observations, the lower left square of (\ref{eq:PsiEqualsTheta}) commutes. This concludes the proof the lemma and thus the proof of Proposition \ref{prop:ComputeTheta}.
\end{proof}

We are now in position to conclude the proof of Theorem \ref{thm:main}. 

\begin{proof}[\textit{Proof of Theorem \ref{thm:main}}]
Let $L$ be a colored link and let $W$ be the exterior of a pushed-in totally connected C-complex for $L$. Recall that $i$ denotes the inclusion induced map from $H_2(W;\Lambda_S)$ to $H_2(W,\partial W;\Lambda_S)$ and that given torsion elements $[x]$ and $[y]$ in $ H_1(X_L;\Lambda_S) \cong H_1(\partial W;\Lambda_S) \cong \coker(i)$, there exists $x_0,y_0 \in H_2(W;\Lambda_S)$ and $r,s \in \Lambda_S$ such that $i(x_0)=r x$ and $i(y_0)=s y$. 
Using Proposition \ref{prop:ReduceToTheta}, we already know that $\op{Bl}(L)([x],[y])=-\theta(x,y)$. Next, Proposition \ref{prop:ComputeTheta} implies that $\theta(x,y)=\psi(x,y)=\frac{1}{r\overline{s}}\lambda(x_0,y_0)$. Summarizing, we have
\begin{equation}
\label{eq:BlanchfieldEqualPsi}
 \op{Bl}(L)([x],[y])=-\theta(x,y)=-\psi(x,y)=-\frac{1}{r\overline{s}}\lambda(x_0,y_0).
 \end{equation}
Note that any choice of $x_0,y_0$ will do since $\lambda$ vanishes on $\ker(i)$; this was already noticed in the definition of $\psi$. Furthermore, note that (\ref{eq:BlanchfieldEqualPsi}) holds independently of the chosen representatives $x$ and $y$ for the classes $[x]$ and $[y]$. 
Indeed if $x$ and $x'$ represent $[x]$, we claim that $\psi(x,y)$ and $\psi(x',y)$ coincide in $Q/\Lambda_S$, i.e. that $\psi(x-x',y)$ lies in $\Lambda_S$; the same proof will hold for the second variable. 
Since $x$ and $x'$ both represent $[x]$, there is a $v$ in $H_2(W;\Lambda_S)$ for which $x-x'=i(v)$. 
Note that $i(r v)=r i(v)$. 
Picking $y_0$ such that $i(y_0)=s y$ and using the definition of $\lambda$, the following equalities prove our claim, since the rightmost term lies in~$\Lambda_S$:
$$ \psi(x-x',y)= \psi(i(v),y)=\frac{1}{r\overline{s}} \lambda(r v, y_0)=\frac{1}{\overline{s}} \langle (\text{PD} \circ i)(y_0),v \rangle=\langle \text{PD}(y),v \rangle.$$
Using \cite[Theorem 1.3]{ConwayFriedlToffoli}, we know that there are bases with respect to which the intersection pairing $\lambda$ on $H_2(W;\Lambda_S)$ is represented by the C-complex matrix $H$ described in the introduction. Furthermore, with respect to the same bases, it was observed in \cite[Section 5.2]{ConwayFriedlToffoli} that the map $i$ is represented by $\overline{H}=H^T$. Consequently, Equation (\ref{eq:BlanchfieldEqualPsi}) can be reformulated as follows. 
Let $n$ denote the rank of the $\Lambda_S$-module $H_2(W;\Lambda_S)$. 
Given $[x],[y] \in TH_1(X_L;\Lambda_S)$, we have $\op{Bl}(L)([x],[y])=-\frac{1}{r\overline{s}} x_0^T H \overline{y_0}$ for any choice of $x_0,y_0 \in \Lambda_S^n$ and $r,s \in \Lambda_S$ such that $\overline{H}x_0=r x$ and $\overline{H}y_0=s y$. Using the notations of the introduction, this can be written as
$$\op{Bl}(L)([x],[y])=-\lambda_H([x],[y]).$$
Up to now, we always supposed that $W$ arose by pushing in a totally connected C-complex. Thus, \emph{a priori}, Theorem \ref{thm:main} only holds for C-complex matrices which arise from totally connected C-complexes. To conclude the proof of Theorem \ref{thm:main}, it therefore only remains to check that the pairing $\lambda_H$ is independent of the choice of a C-complex for $L$.

As explained in~\cite[p.~1230]{CimasoniFlorens} (see also~\cite{CimasoniPotential}), if~$F$ and~$F'$ are two C-complexes for isotopic links, then the corresponding C-complex matrices~$H$
and~$H'$ are related by a finite number of the following two moves:
\[
H\mapsto H\oplus(0) \quad \text{and} \quad H\mapsto \begin{pmatrix}
 H &\xi& 0 \\ \xi^*&\lambda&\alpha\\0&\overline{\alpha}&0 
\end{pmatrix}\,,
\]
with~$\alpha$ a unit of~$\Lambda_S$. In the first case, the~$\Lambda_S$-module~$\Lambda_S^n/\overline{H}\Lambda_S^n$ picks up a free rank~$1$ factor, so its torsion submodule is left
unchanged. It can then be checked that~$\lambda_H$ and~$\lambda_{H\oplus(0)}$ are canonically isometric. In the second case, since~$\alpha$ is a unit in~$\Lambda_S$,
one can assume via the appropriate base change that~$H$ is transformed into~$H\oplus\left(\begin{smallmatrix}0&1\\1&0\end{smallmatrix}\right)$. One can then once again check that the forms associated to these two hermitian matrices are canonically isometric.
\end{proof}

The proof of Theorem \ref{thm:main} is now completed. However, we wish to emphasize an argument which we shall use again later on.
\begin{remark}
\label{rem:LambdaSPresentation}
It follows from \cite[Corollary 3.6]{CimasoniFlorens} (see also  \cite[Theorem 1.1]{ConwayFriedlToffoli}) that $\overline{H}$ presents the $\Lambda_S$-coefficient Alexander module $H_1(X_L;\Lambda_S)$ if $H$ is a C-complex matrix which arises from a totally connected C-complex. However, as we saw in the proof of Theorem \ref{thm:main}, $\text{Tor}_{\Lambda_S}(\Lambda_S^n / \overline{H} \Lambda_S^n)$ is (possibly non-naturally) isomorphic to $TH_1(X_L;\Lambda_S)$ for \emph{any} C-complex matrix $H$. Furthermore, the same argument shows that $\overline{H}$ actually presents $H_1(X_L;\Lambda_S)$ under the weaker hypothesis that $H$ is a C-complex matrix which arises from a \emph{connected} C-complex. Indeed, the transformation $H \mapsto H \oplus (0)$ only arises when one wishes to connect two disconnected components of a C-complex, see \cite[page 1230]{CimasoniFlorens}.
\end{remark}

\section{Applications}
\label{sec:Applications}

In this section, we provide several applications of Theorem \ref{thm:main}. First, in Subsection \ref{sub:NonSing} we give a new proof that the Blanchfield pairing is hermitian. Then, in Subsection \ref{sub:Operations} we give quick proofs of some elementary properties of the Blanchfield pairing. Finally, in Subsection~\ref{sub:Boundary} we apply Theorem \ref{thm:main} to boundary links.

\subsection{The Blanchfield pairing is hermitian}
\label{sub:NonSing}
In this subsection, we prove Corollary \ref{cor:NonSingular}, which states that the Blanchfield pairing is hermitian.
Using Theorem~\ref{thm:main}, this reduces to showing the corresponding statement for $\lambda_H$, where $H$ is any C-complex matrix for $L$. Since this is a purely algebraic statement, we shall prove it in a somewhat greater generality. 
\medbreak
Let~$R$ be an integral domain with involution and let $Q(R)$ be its field of fractions. Given an $R$-module $V$, a pairing $b \colon V \times V \to Q(R)/R$ is \emph{sesquilinear} if it is linear in the first entry and antilinear in the second entry. A sesquilinear pairing $b$ is \emph{non-degenerate} (respectively \emph{nonsingular}) if the adjoint map $ V \to  \overline{\op{Hom}_R(V,Q(R)/R)}, \ 
p  \mapsto  (q \mapsto \lambda(q,p)) $
is a monomorphism (respectively an isomorphism) and  \emph{hermitian} if $\overline{\lambda(w,v)}=\lambda(v,w)$ for any $v,w\in V$.

Let~$H$ be a hermitian~$n\times n$-matrix over~$R$. Given classes~$[v]$ and $[w]$ in $\text{Tor}_{R}(R^n/\overline{H}R^n)$, there exists~$v_0, w_0 \in R^n$ and $r,s \in R$ such that~$r v=\overline{H}v_0$ and $s w=\overline{H}w_0$. Proposition \ref{prop:form} will show that setting 
$$\lambda_H([v],[w]):=\frac{1}{r\overline{s}}v_0^T H \overline{w_0}$$
gives rise to a well-defined, hermitian pairing on $\text{Tor}_{R}(R^n/\overline{H}R^n)$.

%
%

Combining Theorem \ref{thm:main} with 
the following proposition 
immediately implies Corollary~\ref{cor:NonSingular}.

\begin{proposition}\label{prop:form}
The assignment~$(v,w)\mapsto \frac{1}{r\overline{s}}v_0^T H\overline{w_0}$ induces a well-defined pairing
\[
\lambda_H \colon \text{Tor}_{R}(R^n/\overline{H}R^n)\times \text{Tor}_{R}(R^n/\overline{H}R^n)\to Q/R
\]
which is hermitian. 
Furthermore, if~$\det(H)$ is non-zero, then this form coincides with the pairing~$(v,w)\mapsto v^T H^{-1}\overline{w}$.
\end{proposition}

\begin{proof}
Let us first check that this definition is independent of the choice of~$v_0 \in R^n$ and $r \in R$ such that~$r v=\overline{H}v_0$. 
If there are $v_0' \in R^n$ and $r' \in R$ such that $r'v=\overline{H}v_0'$, then
$$ \frac{1}{r'\overline{s}}{v_0'}^T H\overline{w_0}-\frac{1}{r\overline{s}}v_0^T H\overline{w_0}
=\frac{1}{\overline{s}} \left( \overline{H}\left( \frac{v_0'}{r'}-\frac{v_0}{r} \right) \right)^T  \overline{w_0}
=\frac{1}{\overline{s}} \left( v-v \right)^T  \overline{w_0}=0.
$$
\color{black}
A similar argument shows that the definition is independent of the choice of~$w_0,s$ such that~$s w=\overline{H}w_0$. 
Next, let us check that it does not depend on the choice of~$v$ representing the class~$[v]$. 
Any other choice is of the form~$v+\overline{H}u$
where~$u$ lies in $R^n$; since~$r(v+\overline{H}u)=\overline{H}(v_0+r u)$, the element
$$ \frac{1}{r\overline{s}}(v_0+ru)^TH\overline{w}_0-\frac{1}{r\overline{s}}v_0^TH\overline{w}_0=\frac{1}{r\overline{s}}(r u)^TH\overline{w_0}=\frac{1}{\overline{s}}u^TH\overline{w_0}=u^T\overline{w} $$
belongs to~$R$, so the class in~$Q(R)/R$ is indeed well-defined. A similar argument shows that it does not depend on the choice of~$w$ representing the class~$[w]$, thus concluding the proof that~$\lambda_H$ is well-defined. The fact that~$\lambda_H$ is sesquilinear is clear, and it is hermitian since~$H$~is.

To show the second claim, first note that if~$\det(H)$ is non-zero, then~$H$ is invertible over~$Q(R)$ so the equation~$r v=\overline{H}v_0$ is equivalent to~$v_0=r\overline{H}^{-1} v$ (and similarly for~$w_0$). Replacing~$v_0$ and~$w_0$ by these values and using the fact that~$H$ is hermitian, we see that $\lambda_H$ indeed coincides with~$(v,w)\mapsto v^T H^{-1}\overline{w}$. This concludes the proof of the proposition.
\end{proof}

\subsection{Some properties of the Blanchfield pairing}
\label{sub:Operations}
Let~$R$ be an
integral domain with involution. Before dealing with the properties of the Blanchfield pairing, we start by investigating the behavior of $\lambda_H$ under direct sums and multiplication by norms.
\begin{lemma}
\label{lem:Additivity}
Let $H_1,\ldots, H_\mu$ and $H$ be hermitian matrices and let $u$ be a unit of $R$.
\begin{enumerate}
\item Setting $B:=H_1 \oplus \ldots \oplus H_\mu$, one has $\lambda_B=\bigoplus_{i=1}^\mu \lambda_{H_i}.$
\item The pairings $\lambda_{u \overline{u} H}$ and $\lambda_H$ are isometric.
\end{enumerate}
\end{lemma}
\begin{proof}
We start by proving the first assertion. Assume that each $H_i$ is of size $k_i$, set $k:=k_1 + \ldots +k_\mu$ and observe that $R^k/\overline{B} R^k$ is equal to $R^{k_1} / \overline{H_1} R^{k_1} \oplus R^{k_2} / \overline{H_2} R^{k_2} \oplus \ldots \oplus R^{k_\mu} / \overline{H_\mu} R^{k_\mu}$. 

Next, we compute the sum of the $\lambda_{H_i}$. Let $x=x^1 \oplus x^2 \oplus \ldots \oplus x^\mu$ and $y=y^1 \oplus y^2 \oplus \ldots \oplus y^\mu$ be torsion elements in $R^k/\overline{B} R^k$. Relying on the previous paragraph, the $x^i$ and $y^i$ are torsion in $R^{k_i} / \overline{H_i} R^{k_i}$, and so there exists $x_0^i,y_0^i$ and $r_i,s_i \in R$ which satisfy $\overline{H_i} x_0^i=r_i x^i$ and $\overline{H_i} y_0^i=s_i y^i$. Thus, by definition we have
\begin{equation}
\label{eq:ProofAdditivity}
\bigoplus_{i=1}^\mu \lambda_{H_i}(x,y)= \sum_{i=1}^\mu \frac{1}{r_i\overline{s}_i}  (x_0^i)^T H_i \overline{y_0^i}.
 \end{equation}
In order to compute $\lambda_B$ and conclude the proof we proceed as follows. 
First, set $r:=\prod_{i=1}^\mu r_i$ and~$s:=\prod_{i=1}^\mu s_i$. Then define
$x_0$ in $R^k/\overline{B} R^k$ by requiring its i-th component to be equal to $r\cdot r_i^{-1} x_0^i$. 
This way, the $i$-th component of $\overline{B} x_0$ is $\overline{H_i} (r\cdot r_i^{-1} x_0^i)=r\cdot r_i^{-1} \overline{H_i} x_0^i=r x^i$ and thus $\overline{B} x_0=r x$. We can therefore use $x_0$ and $y_0$ to compute $\lambda_B(x,y)$ and we get 
$$ \lambda_B (x,y)=\frac{1}{r\overline{s}} x_0^T B \overline{y_0}=\frac{1}{r\overline{s}} \sum_{i=1}^n (r \cdot r_i^{-1} x_0^i)^T H_i \overline{(s \cdot s_i^{-1} y_0^i)}= \sum_{i=1}^n  \frac{1}{r_i\overline{s}_i}(x_0^i)^T H_i \overline{y_0^i},$$
which agrees with (\ref{eq:ProofAdditivity}). This concludes the proof of the first statement.

To deal with the second statement, first observe that since $u$ is a unit, so are $\overline{u}$ and $u \overline{u}$. Consequently $R^n/\overline{H} R^n$  is equal to $R^n /(u \overline{u}\overline{H}) R^n$ and thus the corresponding torsion submodule supports both the pairings $\lambda_H$ and $\lambda_{u\overline{u}H}$. To prove the assertion, we wish to show that the automorphism $\varphi$ defined by sending $x$ to $u^{-1} x$ provides the desired isometry. To see this, start with torsion elements $x$ and $y$ in the cokernel of $H$ and let $x_0, y_0$ be such that $ \overline{H} x_0=r x$ and $\overline{H}y_0=s y$ with $r,s \in R$.  Since $u$ is a unit, $(u\overline{u})^{-1}$ lies in $R$, and thus $(u \overline{u} \overline{H})((u \overline{u})^{-1} x_0)=r x$ and similarly for $y$. It follows that
$$\lambda_{u \overline{u}H}(x,y)
=\frac{1}{r\overline{s}} ((u \overline{u})^{-1} x_0)^T (u \overline{u} H) \overline{((u \overline{u})^{-1} y_0)} 
=( u \overline{u})^{-1} \frac{1}{r\overline{s}} x_0^T H \overline{y_0}
=( u \overline{u})^{-1} \lambda_H(x,y).
$$
On the other hand, the sesquilinearity of $\lambda_H$ immediately implies that $\lambda_{H} (\varphi(x),\varphi(y))=\lambda_{H} (u^{-1} x,u^{-1} y) =( u \overline{u})^{-1} \lambda_H(x,y)$. Consequently $\lambda_H $ and $\lambda_{u \overline{u} H}$ are isometric, which concludes the proof.
\end{proof}

We can now apply Lemma \ref{lem:Additivity} to obtain some results on the Blanchfield pairing. Before that however, given a C-complex $F$  for a $\mu$-colored link and a sequence~$\eps=(\eps_1,\dots,\eps_\mu)$ of~$\pm 1$'s, we briefly recall some terminology. Pushing curves off $F_i$ in the~$\eps_i$-normal direction for $i=1,\ldots,\mu$ produces a map $i_\varepsilon \colon H_1(F)\to H_1(S^3\setminus F)$. The assignment $\alpha^\varepsilon(x,y):= \op{lk} (i_\eps(x),y)$ gives rise to a bilinear pairing on $H_1(F)$ and thus to a \emph{generalized Seifert matrix} $A^\varepsilon$ for $L$. We refer to \cite{CimasoniPotential, CimasoniFlorens, Cooper} for details.

In the next two propositions, we shall use $\op{Bl}(L)(t_1,\ldots,t_\mu)$ to denote the Blanchfield pairing of a $\mu$-colored link and similarly for the C-complex matrices.

\begin{proposition}
\label{prop:ConnectedSum}
Let $L'=L_1 \cup \ldots \cup L_{\nu-1} \cup L'_\nu$ and $L''=L_\nu'' \cup L_{\nu+1} \cup \ldots \cup L_\mu$ be two colored links. Consider a colored link $L=L_1 \cup \ldots \cup L_\mu $, where $L_\nu$ is a connected sum of $L_\nu'$ and $L_\nu''$ along any of their components. Then $ \op{Bl}(L)(t_1,\ldots,t_\mu)$ is isometric to $\op{Bl}(L')(t_1,\ldots,t_\nu) \oplus \op{Bl}(L'')(t_\nu,\ldots,t_\mu).$
\end{proposition}
\begin{proof}
Denote $\prod_{i > \nu} (1-t_i^{-1})(1-t_i)$ by $u_1$ and $\prod_{i < \nu} (1-t_i^{-1})(1-t_i)$ by $u_2$. Given a C-complex $F'$ for $L'$ and a C-complex $F''$ for $L''$, a C-complex for $L$ is given by the band sum of $F'$ and $F''$ along the corresponding components of $F_\nu'$ and $F_\nu''$. Consequently, $A_F^\varepsilon=A^{\varepsilon'}_{F'} \oplus A^{\varepsilon''}_{F''}$, with $\varepsilon'=(\varepsilon_1,\ldots,\varepsilon_\nu)$ and $\varepsilon''=(\varepsilon_\nu,\ldots,\varepsilon_\mu)$. It follows that a C-complex matrix $H$ for $L$ is given by 
$$ H=u_1 H'(t_1,\ldots,t_\nu) \oplus u_2 H''(t_\nu,\ldots, t_\mu). $$
Denoting these matrices by $H'$ and $H''$, Theorem \ref{thm:main} and the first assertion of Lemma \ref{lem:Additivity} imply that $\op{Bl}(L)$ is isometric to $\lambda_{u_1H'}\oplus \lambda_{u_2 H''}$. Since $u_1$ and $u_2$ are of the form $u \overline{u}$ with $u$ a unit of $\Lambda_S$, the result follows from the second assertion of Lemma \ref{lem:Additivity}.
\end{proof}

\begin{figure}[!htb]
\includegraphics[scale=0.5]{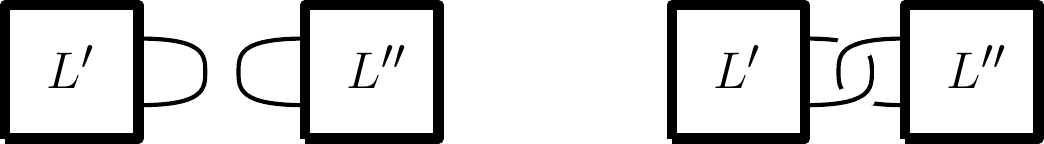}
\caption{Performing a trivial band clasping of the links $L'$ and $L''$}
\label{fig:BandClasp}
\end{figure}

A \emph{trivial band clasping} of two links is the operation depicted in Figure \ref{fig:BandClasp}. A proof similar to the one of Proposition \ref{prop:ConnectedSum} yield the following result.

\begin{proposition}
\label{prop:BandClaspDisjointUnion}
Let $L'=L_1 \cup \ldots \cup L_\nu$ and $L''=L_{\nu+1} \cup \ldots \cup L_\mu$ be colored links with disjoint sets of colors.
\begin{enumerate}
\item  Consider a colored link $L$ obtained by trivially band clasping $L_\nu$ and $L_{\nu+1}$ along any of their components. Then $ \op{Bl}(L)(t)$ is isometric to $\op{Bl}(L')(t') \oplus \op{Bl}(L'')(t'').$
\item Consider the colored link given by the disjoint sum of $L'$ and $L''$. Then $\op{Bl}(L)(t)$ is isometric to $\op{Bl}(L')(t') \oplus \op{Bl}(L'')(t'')$.
\end{enumerate}
\end{proposition}

We conclude this subsection by studying the effect of orientation reversal and taking the mirror image.

\begin{proposition}
\label{prop:Mirror}
Let $L$ be a colored link.
\begin{enumerate}
\item If $\overline{L}$ denotes the mirror image of $L$, then $\op{Bl}(\overline{L})$ is isometric to $-\op{Bl}(L)$.
\item If $-L$ denotes $L$ with the opposite orientation, then $\op{Bl}(-L)$ is isometric to $\op{Bl}(L)$.
\end{enumerate}

\end{proposition}
\begin{proof}
If $F$ is a C-complex for $L$, then the mirror image $F'$ of $F$ is a C-complex for $\overline{F}$. It follows that $H'=-H$. Since these two matrices present the same module, the corresponding torsion submodule supports both $\lambda_H$ and $\lambda_{-H}$. We claim that the automorphism which sends $x$ to $-x$ gives the required isometry. Indeed, if $\overline{H} x_0=r x$ and $\overline{H} y_0=s y$, then $(-\overline{H})x_0=r(-x)$ and $(-\overline{H}) y_0=s (-y)$. Consequently $\lambda_{-H}(-x,-y)$ and $-\lambda_H(x,y)$ are both equal to $-x_0^TH\overline{y_0}$. The result now follows from Theorem \ref{thm:main}. The second assertion follows similarly by noting that a C-complex matrix for $-L$ is given by $\overline{H}$ and by using the fact that $\lambda_H$ is hermitian.
\end{proof} 

\subsection{Boundary links}
\label{sub:Boundary}

An $n$-component \textit{boundary link} is a link $L=K_1 \cup \ldots \cup K_n$ whose $n$ components bound $n$ disjoint Seifert surfaces $F_1, \ldots, F_n$. Set $F=F_1 \sqcup \ldots \sqcup F_n$. Pushing curves off this \emph{boundary Seifert surface} in the negative normal direction produces a homomorphism $i_- \colon H_1(F) \to H_1(S^3 \setminus F)$. The assignment $\theta(x,y):=\lk(i_-(x),y)$ gives rise to a pairing on $H_1(F)$ and to a \textit{boundary Seifert matrix for $L$}, see \cite[p.670]{Ko} for details. Since $H_1(F)$ decomposes as the direct sum of the $H_1(F_i)$, the restriction of $\theta$ to $H_1(F_i) \times H_1(F_j)$ produces matrices $A_{ij}$. For $i \neq j$, these matrices satisfy $A_{ij}=A_{ji}^T$, while  $A_{ii}$ is nothing but a Seifert matrix for the knot $K_i$. Let $g_i$ be the genus of $F_i$, let $I_k$ be the $k \times k$ identity matrix, let $\tau$ be the block diagonal matrix whose diagonal blocks are  $ t_1I_{2g_1}, t_2I_{2g_2},\ldots, t_n I_{2g_n}$ and set $g:=g_1+\ldots + g_n$. We use Theorem \ref{thm:main} in order to give a new proof of the following result, originally due to Hillman \cite[pages 122-123] {HillmanAlexanderIdeals}, see also \cite[Theorem 4.2]{CochranOrr}. 

\begin{theorem}
\label{thm:Boundary}
Let $L=K_1 \cup \ldots \cup K_n$ be a boundary link. Assume that $A$ is a boundary Seifert matrix for $L$ of size $2g$. The Blanchfield pairing of $L$ is isometric to 
\begin{align*}
\label{eq:Boundary}
\Lambda_S^{2g} / (A\tau -A^T) \Lambda_S^{2g} \times \Lambda_S ^{2g} / (A\tau -A^T) \Lambda_S^{2g} &\,\,\rightarrow\,\, Q/\Lambda_S \\
(a,b) &\,\,\mapsto\,\, a^T (A-\tau A^T)^{-1}(\tau-I_{2g})\overline{b}. \nonumber
\end{align*}
\end{theorem}
\begin{proof}
Let $F$ be a boundary Seifert surface which gives rise to $A$. View $F$ as a C-complex for $L$, and use $A_{ij}^\varepsilon$ to denote the restriction of the generalized Seifert matrix $A^\varepsilon$ to $H_1(F_i) \times H_1(F_j)$. If $i \neq j$, since $L$ is a boundary link, $A_{ij}^\varepsilon$ is independent of $\varepsilon$ and is equal to the block $A_{ij}$ of the boundary Seifert matrix $A$. Similarly, for each $\varepsilon$ with $\varepsilon_i=-1$, the restriction of $A^\varepsilon$ to $H_1(F_i) \times H_1(F_i)$ is equal to the block $A_{ii}$. Let $H_i=(1-t_i)A_{ii}^T+(1-t_i^{-1})A_{ii}$ denote the corresponding C-complex matrix for the knot $K_i$ and let $u$ denote $\prod_{j=1}^n(1-t_j)$. The previous considerations show that a C-complex matrix $H$ for $L$ is given by 
$$
H=
\begin{bmatrix}
u\overline{u}(1-t_1)^{-1}(1-t_1^{-1})^{-1}H_1 &u\overline{u} A_{12} & \ldots & u\overline{u} A_{1n} \\
u\overline{u} A_{21} & u\overline{u}(1-t_2)^{-1}(1-t_2^{-1})^{-1}H_2 & \ldots & u\overline{u} A_{2n} \\
\vdots & \ddots & \ddots& \vdots \\
u\overline{u} A_{n1} & u\overline{u} A_{n2} & \ldots & u\overline{u}(1-t_n)^{-1}(1-t_n^{-1})^{-1}H_n
\end{bmatrix}.$$
Since $H_i=(1-t_i^{-1})(A_{ii}-t_iA_{ii}^T)$, the diagonal blocks of $H$ can be rewritten as $u\overline{u} (1-t_i)^{-1} (A_{ii}-t_iA_{ii}^T)$. Using the equation $A_{ij}=A_{ji}^T$, we see that a C-complex matrix for $L$ is given by
\begin{equation}
\label{eq:BoundaryComput}
H=u\overline{u} (I_{2g}-\tau)^{-1} (A-\tau A^T).
\end{equation}
It follows that $H^T=u \overline{u}(A^T-A\tau )(I_{2g}-\tau)^{-1}$. Since $u$ is a unit of $\Lambda_S$ and $(I_{2g}-\tau)^{-1}$ is an automorphism of $\Lambda_S^{2g}$, the module presented by $\overline{H}=H^T$ is canonically isomorphic to the module presented by $A\tau -A^T$. As the isomorphism is induced by the identity of $\Lambda_S^{2g}$, we shall slightly abuse notations and consider these modules as ``equal", see the second left vertical arrow in~(\ref{eq:BoundaryDiagram}).

Next, we claim that $\Lambda_S^{2g}/\overline{H} \Lambda_S^{2g}$ is $\Lambda_S$-torsion. Band clasp trivially $F_1$ with $F_2$, $F_2$ with $F_3$, $F_i$ with $F_{i+1}$ and finally $F_{n-1}$ with $F_n$. The result is a link $L'$ which bounds a connected C-complex $F'$ for which the associated C-complex matrix is also $H$. Since $L$ has pairwise vanishing linking numbers, $L'$ does not. Consequently, using the Torres formula, the Alexander polynomial of $L'$ is non-zero and thus its Alexander module is torsion. As we saw in Remark \ref{rem:LambdaSPresentation}, if a C-complex matrix $H$ arises from a \emph{connected} C-complex, $\overline{H}$ presents the $\Lambda_S$-localized Alexander module. Thus $\overline{H}$ presents the torsion module $H_1(X_{L'};\Lambda_S)$ and the claim follows.

 Now consider the following diagram: 
\begin{equation}
\label{eq:BoundaryDiagram}
\xymatrix@R0.4cm@C0.2cm{TH_1(X_L;\Lambda_S)\times TH_1(X_L;\Lambda_S)\ar[rrrrrrrrrrr]^-{\op{Bl}(L)}\ar[dd]^\cong 
&&&&&&&&&&&&Q/\Lambda_S\ar[dd]^=\\
 \\
 \frac{\Lambda_S^{2g}}{\overline{H}\Lambda_S^{2g}} \times   \frac{\Lambda_S^{2g}}{\overline{H}\Lambda_S^{2g}} \ar[rrrrrrrrrrr]^-{(a,b) \mapsto -a^T H^{-1}\overline{b}}\ar[dd]^{=}
  &&&&&&&&&&&&Q/\Lambda_S\ar[dd]^=\\\\
\frac{\Lambda_S^{2g}}{(A\tau -A^T)\Lambda_S^{2g}} \times  \frac{\Lambda_S^{2g}}{(A\tau -A^T)\Lambda_S^{2g}}  \ar[rrrrrrrrrrr]^-{(a,b) \mapsto a^T (u\overline{u})^{-1}(A-\tau A^T)^{-1}(\tau-I_{2g})\overline{b}}\ar[dd]^{(a,b)\mapsto (u^{-1}a,u^{-1}b)} 
&&&&&&&&&&&&Q/\Lambda_S\ar[dd]^=\\\\
\frac{\Lambda_S^{2g}}{(A\tau-A^T)\Lambda_S^{2g}}  \times  \frac{\Lambda_S^{2g}}{(A\tau -A^T)\Lambda_S^{2g}}  \ar[rrrrrrrrrrr]^-{(a,b) \mapsto a^T (A-\tau A^T)^{-1}(\tau-I_{2g})\overline{b}}
 &&&&&&&&&&&&Q/\Lambda_S.\\
}
\end{equation}
The top square commutes by Theorem~\ref{thm:main}. Note that Proposition \ref{prop:form} ensures that $\lambda_H(a,b)=-a^T H^{-1} \overline{b}$: indeed we argued above that $\Lambda_S^{2g}/\overline{H}\Lambda_S^{2g}$ is torsion. The middle rectangle, commutes thanks to (\ref{eq:BoundaryComput}). Finally, the commutativity of the bottom square follows from a direct computation.
\end{proof}

\appendix

\section{Proof of Lemma~\ref{lem:BigDiagram}}

For the reader's convenience, we recall both the set-up and the statement of Lemma~\ref{lem:BigDiagram}. Given a commutative ring $R$, consider the following commutative diagram of cochain complexes of $R$-modules whose columns and rows are assumed to be exact:
\begin{equation}
\label{eq:NineAppendix}
\xymatrix@R0.5cm{
  &0 \ar[d]&  0\ar[d]&  0\ar[d]&  \\
0 \ar[r] &A\ar[d] \ar[r]& B \ar[d]_{v_B}\ar[r]^{h_B}& C \ar[r]\ar[d]& 0 \\
0 \ar[r] &D \ar[r]^{h_D}\ar[d]_{v_D}& E \ar[r]\ar[d]_{v_E}& F \ar[r]\ar[d]& 0 \\
0 \ar[r] &H \ar[r]^{h_H}\ar[d]& J \ar[r]^{h_J}\ar[d]& K \ar[r]\ar[d]& 0  \\
  &0 &  0&  0.& 
}\end{equation} 
If $0 \rightarrow S \rightarrow T \rightarrow U \rightarrow 0$ is one of the short exact sequences of cochain complexes in~(\ref{eq:NineAppendix}), we shall use~$\delta_{U}^v$ (resp.~$\delta_{U}^h$) to denote the connecting homomorphism $H^*(U) \rightarrow H^{*+1}(S)$ if the sequence is depicted vertically (resp. horizontally). For instance, there are connecting  homomorphisms $\delta_K^v \colon H^*(K) \to H^{*+1}(C)$ and $\delta_K^h \colon H^*(K) \to H^{*+1}(H)$.

Just as in Section~\ref{sec:ProofMain}, we use the same notation for cochain maps as for the homomorphisms they induce on cohomology. Furthermore, we shall write $H^*(D) \to H^*(J)$ for the homomorphism induced by any composition of the cochain maps from $D$ to $J$. Also, $H^*(J) \to H^{*+1}(C)$ will denote the composition of the connecting homomorphism $\delta^h_J \colon H^*(J) \to H^{*+1}(B)$ with the homomorphism $h_B \colon H^*(B) \to H^*(C)$. Alternatively, the latter map can also be described as the composition of the homomorphism induced by the cochain map $h_J \colon H^*(J) \to H^*(K)$ with the connecting homomorphism $\delta_K^v \colon H^*(K) \to H^{*+1}(C)$. 

Furthermore, as we argued in the discussion leading to the statement of Lemma~\ref{lem:BigDiagram}, the connecting homomorphism $\delta_K^v$ induces a well-defined map $\frac{ H^{*-1}(K)}{\ker(\delta_K^h)} \to \frac{H^*(C)}{\im( H^{*-1}(J) \to  H^*(C))}$, which we also denote by $\delta_K^v$. Additionally, there are well-defined homomorphisms $(\delta_K^h)^{-1}$ and~$v_B^{-1}$, whose definitions we now recall, referring to Section~\ref{sec:ProofMain} for details.
\begin{enumerate}
\item There is a homomorphism $(\delta_K^h)^{-1}$ from $v_D \ker( H^*(D)  \to H^*(J))$ to $ H^{*-1}(K)/\ker(\delta_K^h)$. More precisely, $(\delta_K^h)^{-1}(v_D([x]))$ is defined as the class of $[k]$ in $H^{*-1}(K)/\ker(\delta_K^h)$ for any $[k]$ in $H^{*-1}(K)$ such that $\delta_K^h([k])=v_D([x])$.
\item There is a homomorphism $v_B^{-1}$ from  $h_D \ker(  H^*(D) \to H^*(J))$ to $\frac{ H^*(B)}{\ker(v_B)}$. More precisely, $v_B^{-1}(h_D([x]))$ is defined as the class of $[b]$ in $\frac{H^*(B)}{\ker(v_B)}$ for any $[b]$ in $H^*(B)$ such that $v_B([b])=h_D([x])$.
\end{enumerate}
The aim of this appendix is to prove Lemma~\ref{lem:BigDiagram} which states that the following diagram anticommutes:
\begin{equation}
\label{eq:9Lemma}
 \xymatrix@R0.5cm{
\ker(  H^m(D)  \to  H^m(J)) \ar[r]^{  v_D}\ar[d]^{h_D} & v_D \ker(  H^m(D)  \to  H^m(J)) \ar[d]^{(\delta_K^h)^{-1}}\\
h_D \ker(  H^m(D)  \to  H^m(J))\ar[d]^{v_B^{-1}} & \frac{H^{m-1}(K)}{\ker(\delta_K^h)}\ar[d]^{\delta_K^v} \\
\frac{ H^m(B)}{\ker(v_B)} \ar[r]^{h_B}& \frac{ H^m(C)}{\im(H^{m-1}(J) \to  H^m(C))}.
}
\end{equation}
\begin{proof}[Proof of Lemma~\ref{lem:BigDiagram}]
The proof is structured as follows: first, we set up some notation, then we compute $h_B \circ v_B^{-1} \circ h_D$ and, finally, we show that $\delta_K^v \circ (\delta_K^h)^{-1}\circ v_D $ yields the same result.

Our first task is to write out explicitly the short exact sequences of cochain complexes displayed in~(\ref{eq:NineAppendix}). We restrict our attention to the degrees of interest (namely $m$ and $m-1$) and omit the trivial modules which ought to appear on the extremities of the exact rows and columns. The result is the following commuting cube of $R$-modules in which the rows and columns are exact:
\begin{equation}
\label{eq:Cube}
 \xymatrix@R0.5cm{
A^{m-1} \ar[rr]\ar[rd] \ar[dd]&& B^{m-1} \ar[rr]\ar[rd] \ar[dd]|\hole&& C^{m-1}\ar[dd]|\hole\ar[rd] & \\
&A^m \ar[rr]\ar[dd] &&  B^m \ar[rr]\ar[dd]&&  C^m\ar[dd]  \\
D^{m-1} \ar[rr]|\hole\ar[rd] \ar[dd] && E^{m-1} \ar[dd]|\hole \ar[rr]|\hole\ar[rd] && F^{m-1}\ar[dd]|\hole \ar[rd]  & \\
&D^m \ar[rr]\ar[dd]&& E^m \ar[rr]\ar[dd]&& F^m\ar[dd]  \\
H^{m-1} \ar[rr]|\hole\ar[rd] && J^{m-1}\ar[rd] \ar[rr]|\hole&& K^{m-1}\ar[rd] & \\
&H^m \ar[rr] && J^m \ar[rr]&& K^m. \\}
\end{equation}
Although the maps of this cube are not labeled, we systematically use the following conventions. Firstly, the codifferential of a cochain complex $T$ will be denoted by $\delta_T$. Secondly, cochain maps are indexed by their domain and are named $h$ (resp. $v$) if they are horizontal (resp. vertical). For instance, in the lower right corner of~(\ref{eq:Cube}), the horizontal map is denoted by $h_J$, the vertical map is denoted by $v_F$ and the diagonal map is denoted by~$\delta_K$. Additionally, recall that we use the same notation for cochain maps as for the homomorphisms they induce on cohomology.

Next, we move on to the second step of the proof: since our goal is to show that the equality
$$ h_B \circ v_B^{-1} \circ h_D([x])=- \delta_K^v \circ (\delta_K^h)^{-1}\circ v_D([x])$$
holds for all $[x] \in \ker(H^m(D) \to H^m(J))$, we now describe the map $h_B \circ v_B^{-1} \circ h_D$. Let $x \in D^m$ be a cocycle representing a class $[x] \in \ker(H^m(D) \to H^m(J))$. As we saw in Section~\ref{sec:ProofMain}, there exists $[b]$ in $H^m(B)$ such that $v_B([b])=h_D([x])$. Fixing once and for all such a $[b]$, the definition of $v_B^{-1}$ implies that $v_B^{-1} \circ h_D([x])$ is equal to the class of $[b]$ in $\frac{H^m(B)}{\ker(v_B)}$. We deduce that $h_B \circ v_B^{-1} \circ h_D([x])=h_B([b])$. 

To carry out the third step of the proof, we must compute $\delta_K^v \circ (\delta_K^h)^{-1}\circ v_D([x])$. Consequently, we briefly recall the definition of connecting homomorphisms.

\begin{remark}
\label{rem:Connecting}
Given a short exact sequence $0 \to S \stackrel{j}{\to} T \stackrel{\pi}{\to} U \to 0$ of cochain complexes, the connecting homomorphisms $\delta_{\text{conn}} \colon H^m(U) \to H^{m+1}(S)$ are defined as follows. Since $\pi$ is surjective, pick any $t \in T^m$ such that $\pi(t)=u$ is a cocycle representing a cohomology class~$[u]$ in $H^m(U)$, and set $\delta_{\text{conn}} ([u]):=[s]$, where $s \in S^{m+1}$ is the (unique) cocycle satisfying $j(s)=\delta_T(t)$. It is well known that $\delta_{\text{conn}}$ is well-defined.
\end{remark}
In order to compute $\delta_K^v \circ (\delta_K^h)^{-1}\circ v_D([x])$, we first compute $(\delta_K^h)^{-1}\circ v_D([x])$. Since $h_D([x])-v_B([b])$ vanishes in cohomology (by definition of $[b]$), there is a cochain $e\in E^{m-1}$ such that~$\delta_E (e)=h_D(x)-v_B(b)$.

\begin{claim}
The class of $h_J(v_E([e]))$ in $\frac{H^{m-1}(K)}{\ker(\delta^h_K)}$ is equal to $(\delta_K^h)^{-1} \circ v_D([x])$.
\end{claim}
\begin{proof} 
By definition of $(\delta_K^h)^{-1}$, it is enough to verify that $\delta_K^h(h_J(v_E([e])))=v_D([x])$. To check this, recall from Remark~\ref{rem:Connecting} that we must show that $h_J(v_E(e))$ is a cocyle and that $\delta_J(v_E(e))=h_H(v_D(x))$. To check that $h_J(v_E(e))$ is indeed a cocycle, we use successively the commutativity of~(\ref{eq:NineAppendix}), the definition of $e$ and the exactness of the lines in~(\ref{eq:NineAppendix}) to get
$$\delta_K(h_J(v_E(e)))=v_F(h_E(\delta_E(e)))=v_F(h_E(h_D(x)-v_B(b)))=-v_F(h_E(v_B(b))).$$
Using once again the commutativity of~(\ref{eq:NineAppendix}) and the exactness of its lines, we deduce the desired result, namely that
$$\delta_K(h_J(v_E(e)))=-v_F(h_E(v_B(b)))=-h_J(v_E(v_B(b)))=0.$$
 Next, we check the equality $\delta_J(v_E(e))=h_H(v_D(x))$. This verification is carried out by using successively the commutativity of~(\ref{eq:Cube}), the definition of $e$, the exactness of the columns in~(\ref{eq:Cube}) and the commutativity of~(\ref{eq:Cube}):
 $$\delta_J(v_E(e))=v_E(\delta_E(e))=v_E(h_D(x))-v_E(v_B(b))=v_E(h_D(x))=h_H(v_D(x)).$$
This concludes the proof of the claim.
\end{proof}
 
The conclusion of the lemma will promptly follow from the next claim:

\begin{claim}
The class of $-h_B([b])$ in $\frac{ H^m(C)}{\im(H^{m-1}(J) \to  H^m(C))}$ is equal to $\delta_K^v \circ (\delta_K^h)^{-1} \circ v_D([x])$.
\end{claim}
\begin{proof}
We will show that the announced equality holds without having to pass to the quotient. To check this assertion, recall from Remark~\ref{rem:Connecting} that we must find a cochain $f \in F^{m-1}$ such that $v_F(f)$ is a cocycle representing $(\delta_K^h)^{-1} \circ v_D([x])$; furthermore $f$ must also satisfy $\delta_F(f)=-v_C(h_B(b))$. 

We claim that $h_E(e)$ can be taken to play the role of $f$. We first check that $h_E(e)$ is such that $v_F(h_E(e))$ is a cocycle representing $(\delta_K^h)^{-1} \circ v_D([x])$. Since we proved in the previous claim that $(\delta_K^h)^{-1} \circ v_D([x])$ is (the class of) the cohomology class of $h_J(v_E(e))$, it is actually enough to show that $v_F(h_E(e))= h_J(v_E(e))$. This is immediate from the commutativity of~(\ref{eq:Cube}). Finally, we show that $\delta_F(h_E(e))=-v_C(h_B(b))$. This follows from the commutativity of~(\ref{eq:Cube}), the definition of $e$, the exactness of the rows in~(\ref{eq:Cube}), and the commutativity of~(\ref{eq:Cube}):
$$ \delta_F(h_E(e))=h_E(\delta_E(e))=h_E(h_D(x))-h_E(v_B(b)) =-h_E(v_B(b)) =-v_C(h_B(b)).$$
This concludes the proof of the claim.
\end{proof}
Summarizing, we have just shown that $-[h_B(b)]$ represents $\delta_K^v \circ (\delta_K^h)^{-1} \circ v_D([x])$. Since the second step of the proof consisted in showing that $h_B([b])$ represents $h_B \circ v_B^{-1} \circ h_D([x])$, the proof of the lemma is concluded.
\end{proof}

\bibliographystyle{plain}
\bibliography{BiblioBlanchfield2}

\end{document}